\documentclass{amsart}
\usepackage{amssymb}
\usepackage{amsmath}
\usepackage{amsfonts}
\usepackage{amsthm}

\numberwithin{equation}{section}

\newtheorem{theorem}{Theorem}[section]
\newtheorem{lemma}{Lemma}[section]
\newtheorem{proposition}{Proposition}[section]

\def\R{\mathbb{R}}

\def\N{\mathbb{N}}

\def\supp{\mathop{\rm supp}\nolimits}

\begin{document}

\title[Dispersive limit from the Kawahara to the KdV equation]{Dispersive limit from the Kawahara to the KdV equation}
\author{ Luc Molinet }
\address{ Universit\'e Fran\c cois Rabelais Tours, F\'ed\'eration Denis Poisson-CNRS,
 Parc Grandmont, 37200 Tours,  France}
 \email{luc.molinet@lmpt.univ-tours.fr}
\thanks{}
\author{Yuzhao Wang}
  \address{Department of Mathematics and Physics, North China Electric Power University, Beijing 102206, China}
  \email{wangyuzhao2008@gmail.com}
\subjclass[2000]{ 35Q53}
\keywords{KdV equation, Kawahara equation, dispersive limit. }

\begin{abstract}
We investigate the limit behavior of the solutions to the Kawahara equation
$$
u_t +u_{3x} +\varepsilon u_{5x} + u u_x =0 \quad , \, Ê \varepsilon>0
$$
as $ \varepsilon  \to 0 $. In this equation, the terms  $ u_{3x}Ê$ and $ \varepsilon u_{5x} $ do compete together and do cancel each other  at frequencies of order $ 1/\sqrt{\varepsilon} $. This prohibits the use of a standard dispersive approach for this problem. Nervertheless, by combining different dispersive approaches according to the range of spaces frequencies, we succeed in proving that the solutions to this equation converges in $ C([0,T];H^1(\R)) $ towards the solutions of the
 KdV equation for any fixed $ T>0$.
\end{abstract}

\maketitle

\section{Introduction and main results}
\subsection{Introduction}
In this paper we are interested in the limit behavior of the solutions to the Kawahara equation
$$
(K_\varepsilon) \hspace*{2cm} u_t +u_{3x} +\varepsilon u_{5x} + u u_x =0 , \quad (t,x)\in \R^2 , \; \varepsilon>0,
$$
as the positive coefficient $ \varepsilon\to 0 $. \\
Our goal is to prove  that they converge in a strong sense towards  the solutions of the KdV equation
\begin{equation}\label{KdV}
u_t +u_{3x} + u u_x =0 , \quad (t,x)\in \R^2 .
\end{equation}
This study can be seen as   a peculiar case of the following  class of limit behavior problems :
 \begin{equation}\label{general}
  \partial_t u +\partial_x \Bigl( L_1-\varepsilon L_2\Bigr) u +N_1(u)+\varepsilon N_2(u)=0 \; ,
 \end{equation}
 where $ u \, :\, \R \to \R $,   $L_1 $ and  $L_2 $ are speudo-differential operators with Fourier symbols $ |\xi|^{\alpha_1}$  and $ |\xi|^{\alpha_2}$ with $ 0<\alpha_1<\alpha_2 $ and  $ N_1 $ and $ N_2 $ are polynomial functions that depends on   $u $, its derivatives and possibly on the image of $ u $ by some
  speudo-differential operator (as for instance the Hilbert transform)  .  Note that the dispersive limits from the Benjamin equation or  some higher-order BO equations derived in \cite{CGK} towards  the Benjamin-Ono equation enter this class.

In this class of limit behavior problems, the main difficulty comes from the fact that the dispersive terms $ \partial_x L_1 u  $ and $ \varepsilon \partial_x L_2 u  $ do compete together. As one can easily check, the derivatives of the associated phase function $ \phi(\xi)=\xi |\xi|^{\alpha_1} (1-\varepsilon |\xi|^{\alpha_2-\alpha_1}) $ does vanish at frequencies of order $\varepsilon^{-\frac{1}{\alpha_2-\alpha_1}}$. This will make classical dispersive estimates as Strichartz estimates, global Kato smoothing effect or maximal in time estimate, not uniform in $ \varepsilon$. Therefore it is not clear to get even boundedness uniformly in $ \varepsilon $ of the solutions to \eqref{general}  by classical dispersive resolution methods.

On the other hand, by using only energy estimates that do not take into account the dispersive terms, we can see immediately that the solutions to $(K_\varepsilon) $ will stay bounded in $ H^s(\R)$, uniformly in $ \varepsilon $,  providing we work in Sobolev spaces $ H^s(\R ) $ with index $ s>3/2$. Moreover, using for instance Bona-Smith argument, we could prove the convergence of the solution of
$(K_\varepsilon) $   to the ones of \eqref{KdV} in $ C([0,T] ; H^s(\R)) $ with $ T=T(\|u(0)\|_{H^s} $ and $ s>3/2 $. However this approach is far to be satisfactory since it does not use at all the dispersive effects. Moreover, the KdV and Kawahara equations are known to be well-posed in low indices Sobolev spaces (see for instance  \cite{Bourgain1993}, \cite{KPVJAMS96}, \cite{Kato} )  and one can ask wether such convergence result does hold in those spaces. In this work we make a first step in this direction by proving that this convergence result holds in $ H^s(\R) $ with $ s\ge 1 $. Note that $ H^1(\R ) $ is a natural space for this problem since it is the energy space for the KdV equation.  Our main  idea is to combine different dispersive method according to the area of frequencies we consider. More precisely, we will use a Bourgain's approach (cf. \cite{Bourgain1993}, \cite{Gi}) outside the area $ D_\varepsilon $ where the first derivative of the phase function $ \phi' $ does vanish whereas we will use  Koch-Tzvetkov approach (cf. \cite{KT1}) in $ D_\varepsilon $.
 Indeed, noticing that $ \phi'' $ does not vanish in this area, the Strichartz estimate are valid uniformly in $ \varepsilon $ on  $D_\varepsilon $ so that we can apply Koch-Tzvetkov approach. On the other hand, outside $ D_\varepsilon $ one can easily see that one has a strong resonance relation at least for the worst interactions, namely the high-low interactions. Indeed, assuming that $
|\xi_1|>>|\xi_2| $, by the mean-value theroem, it holds
$$
|\phi_\varepsilon(\xi_1+\xi_2)-\phi_\varepsilon(\xi_1)-\phi_\varepsilon(\xi_2)|
\sim |\phi_\varepsilon'(\xi_1) \xi_2 - \phi_\varepsilon(\xi_2) |\sim
|\phi_\varepsilon'(\xi_1) \xi_2| \sim |\xi^2(3-5\varepsilon \xi^2)
\xi_2| \gtrsim \xi^2 |\xi_2| ,
$$
where $ \xi=\xi_1+\xi_2 $ is the output frequency and $ \phi_\varepsilon(\xi)=\xi^3-\varepsilon \xi^5 $ is the phase function associated with the $ (K_\varepsilon) $. It is worth noticing that this  resonance relation is similar to the one of the  KdV equation that reads  $ (\xi_1+\xi_2)^3-(\xi_1)^3-(\xi_2)^3=3\xi\xi_1\xi_2 $.
 To rely on this strong resonance relation even when one of the input frequency belongs to $ D_\varepsilon $ we will make use of the fact that any $ H^1 $-solution to $(K_\varepsilon)$ must belong to some Bourgain's space with time regularity one.
  \subsection{Main results}
  \begin{theorem}\label{theo1}
Let $ s\ge 1 $, $ \varphi\in H^s(\R) $,  $ T>0 $ and $ \{\varepsilon_n\}_{n\in\N} $ be a decreasing sequence of real numbers converging to $ 0 $.  The sequence $ u_n\in C(\R;H^s(\R)) $ of solutions to  $(K_\varepsilon) $ emanating from $ \varphi $  satisfies
 \begin{equation}\label{to1}
u_{n} \to u \mbox{ in }C([0,T];H^{s}(\R))
\end{equation}
where $ u\in C(\R;H^{s}(\R)) $ is the unique solution to the KdV equation  \eqref{KdV}  emanating from $\varphi$.
 \end{theorem}
 Theorem 1 is actually a direct consequence of the fact that the Cauchy problem associated with $(K_\varepsilon) $ is well-posed in
 $ H^s(\R) $, $ s\ge 1$,  uniformly in $  \varepsilon\in ]0,1[ $ in the following sense
 \begin{theorem}\label{theo2}
Let $ s\ge 1 $ and  $ \varphi\in H^s(\R) $. There exists $ T=T(\|\varphi\|_{H^1})\in]0,1[ $ and  $ C>0 $ such that for any $  \varepsilon\in ]0,1[ $ the solution $ u_\varepsilon \in C(\R;H^1(\R)) $ to $(K_\varepsilon) $ satisfies
 \begin{equation}\label{to3}
 \sup_{t\in [0,T]} \| u_\varepsilon (t)\|_{H^s} \le C \|\varphi\|_{H^s}
\end{equation}
Moreover, for any $ R>0 $, the family of  solution-maps $S_{K_\varepsilon}  \, :\,  \varphi \mapsto u_\varepsilon $, $ \varepsilon\in ]0,1[ $, from $ B(0,R)_{H^s}  $ into $  C([0,T(R)]; H^s(\R)) $ is equi-continuous, i.e. for any sequence $ \{\varphi_n\} \subset   B(0,R)_{H^s}  $ converging to $ \varphi $ in $ H^s(\R) $ it holds \begin{equation}\label{to4}
\lim_{n\to 0}  \sup_{\varepsilon\in ]0,1[}\|S_{K_\varepsilon} \varphi - S_{K_\varepsilon} \varphi_n\|_{L^\infty(0,T(R); H^s(\R))}=0 \; .
 \end{equation}
  \end{theorem}
 \subsection{Notation}
 For any positive numbers $a$ and $b$, the notation $a \lesssim b$ means that there exists a positive constant
$c$ such that $a \le c b$. We also denote $a \sim b$ when $a \lesssim b$ and $b \lesssim a$.
 Moreover, if $\alpha \in \mathbb R$, $\alpha_+$, respectively $\alpha_-$, will denote a number slightly greater, respectively lesser, than $\alpha$.

For $u=u(x,t) \in \mathcal{S}(\mathbb R^2)$,
$\mathcal{F}u=\widehat{u}$ will denote its space-time Fourier
transform, whereas $\mathcal{F}_xu=(u)^{\wedge_x}$, respectively
$\mathcal{F}_tu=(u)^{\wedge_t}$, will denote its Fourier transform
in space, respectively in time. For $s \in \mathbb R$, we define the
Bessel and Riesz potentials  $J^s_x$ and $D_x^s$, by
\begin{displaymath}
J^s_xu=\mathcal{F}^{-1}_x\big((1+|\xi|^2)^{\frac{s}{2}}
\mathcal{F}_xu\big) \quad \text{and} \quad
D^s_xu=\mathcal{F}^{-1}_x\big(|\xi|^s \mathcal{F}_xu\big).
\end{displaymath}

 We will  need a Littlewood-Paley analysis. Let $\psi\in C^\infty_0(\R)$ be an even function  such that $\psi\geq 0$, $\supp \psi\subset [-3/2,3/2]$, $\psi\equiv 1$ on $[-5/4,5/4]$. We set $\eta_0:=\psi $ and  for all $k\in \N^* $, $\eta_{2^k}(\xi):=\psi(2^{-k} \xi)-\psi(2^{-k+1} \xi)$,  $ \eta_{\le 2^k}:=\psi(2^{-k} \cdot)=\sum_{j=0}^k \eta_{2^j} $ and  $ \eta_{\ge 2^k}:=1-\psi(2^{k-1}\cdot)=1-\eta_{\le 2^{k-1}}$.  The Fourier multiplicator operators by $ \eta_{2^j} $, $ \eta_{\le 2^j} $  and $ \eta_{\ge 2^j} $  will be denoted respectively by $ P_{2^j} $,  $P_{\le 2^j}  $  and $P_{\ge 2^j}  $, i.e. for any $ u\in L^2(\R)$
  $$
 \widehat{P_{2^j} u}:= \eta_{2^j} \widehat{u} , \quad Ê \widehat{P_{\le 2^j} u}:=\eta_{\le 2^j} \hat{u}
 \quad \mbox{ and }Ê \widehat{P_{\ge 2^j} u}:=\eta_{\ge 2^j} \hat{u}
 \;  .
$$

Note that, to simplify the notations, any summations over capitalized variables such as $N$ are presumed to be dyadic with $N \ge 1 $, \textit{i.e.},
these variables range over numbers of the form $2^k $, $k \in \mathbb Z_{+}$. $P_+$ and $P_-$  will denote the projection on respectively the
positive and the negative Fourier frequencies.

Finally, we denote by $ U_\varepsilon(t):=e^{-t(\partial_x^3+\varepsilon \partial_x^5)} $ the free evolution associated with the linear part of $ (K_\varepsilon) $.
\subsection{Function spaces} \label{spaces}
For $1 \le p \le \infty$, $L^p(\mathbb R)$ is the usual Lebesgue
space with the norm $\|\cdot\|_{L^p}$, and for $s \in \mathbb R$ ,
the real-valued Sobolev spaces $H^s(\mathbb R)$  denote the spaces of all real-valued functions with the usual
norms
\begin{displaymath}
\|\varphi\|_{H^s}=\|J^s_x\varphi\|_{L^2} \; .
\end{displaymath}
If $f=f(x,t)$ is a function defined for $x \in
\mathbb R$ and $t$ in the time interval $[0,T]$, with $T>0$, if $B$
is one of the spaces defined above, $1 \le p \le \infty$ and $1 \le q \le \infty$, we will
define the mixed space-time spaces $L^p_TB_x$,
$L^p_tB_x$, $L^q_xL^p_T$ by the norms
\begin{displaymath}
\|f\|_{L^p_TB_x} =\Big(
\int_0^T\|f(\cdot,t)\|_{B}^pdt\Big)^{\frac1p} \quad , \quad
\|f\|_{L^p_tB_x} =\Big( \int_{\mathbb R}\|f(\cdot,t)\|_{B}^pdt\Big)^{\frac1p},
\end{displaymath}
and
\begin{displaymath}
\|f\|_{L^q_xL^p_T}= \left(\int_{\mathbb R}\Big( \int_0^T|f(x,t)|^pdt\Big)^{\frac{q}{p}}dx\right)^{\frac1q}.
\end{displaymath}
For $s$, $b \in \mathbb R$, we introduce the Bourgain spaces
$X^{s,b}_{\epsilon}$ related to the linear part of $(K_\varepsilon) $  as
the completion of the Schwartz space $\mathcal{S}(\mathbb R^2)$
under the norm
\begin{equation} \label{Bourgain}
\|v\|_{X^{s,b}_{\epsilon}} := \left(
\int_{\mathbb{R}^2}\langle\tau-\phi_\varepsilon(\xi)\rangle^{2b}\langle \xi\rangle^{2s}|\widehat{v}(\xi, \tau)|^2
d\xi d\tau \right)^{\frac12},
\end{equation}
where $\langle x\rangle:=1+|x|$. We will also use a dyadic version
of those spaces introduced in \cite{Tat} in the context of wave maps.
For $s$, $b \in \mathbb R$, $1 \le q \le
\infty$, $X^{s,b,q}_{\epsilon}$ will denote the completion of the
Schwartz space $\mathcal{S}(\mathbb R^2)$ under the norm
\begin{equation} \label{Bourgain2}
\|v\|_{X^{s,b,q}_{\epsilon}} := \left( \sum_{k\ge 0}\Big(\sum_{j\ge 0}\langle 2^k
\rangle^{sq}\langle
2^j \rangle^{bq}\|P_{2^k}(\xi) P_{2^j}(\tau-\phi_\varepsilon(\xi))\widehat{v}(\xi,\tau)\|_{L^2_{\tau,\xi}}^q\Big)^{\frac2q}
\right)^{\frac12}.
\end{equation}
Moreover, we define a localized (in time) version of these spaces.
Let $T>0$ be a positive time and $Y=X^{s,b}_{\epsilon}$ or
$Y=X^{s,b,q}_{\epsilon}$. Then, if $v: \mathbb R \times
]0,T[\rightarrow \mathbb R$, we have that
\begin{displaymath}
\|v\|_{Y_{T}}:=\inf \{\|\tilde{v}\|_{Y} \ | \ \tilde{v}: \mathbb R
\times \mathbb R \rightarrow \mathbb C, \ \tilde{v}|_{\mathbb R
\times ]0,T[} = v\}.
\end{displaymath}

\section{Uniform estimates far from the stationary point of the phase function}\label{section2}
 As we explained in the introduction, it is crucial that the first and the second derivatives of the phase function $ \phi_\varepsilon(\xi)=\xi^3-\varepsilon \xi^5 $ do not cancel
  exactly at the same point. Indeed, $ \phi_\varepsilon'(\xi)=0 \Leftrightarrow |\xi|=\sqrt{\frac{3}{5 \varepsilon}} $ while , $ \phi_\varepsilon''(\xi)=0 \Leftrightarrow |\xi|=\sqrt{\frac{3}{10 \varepsilon}} $. Consequently,
  we introduce the following smooth Fourier projectors
  $$
  \widehat{P_{A_\varepsilon} f} = \Bigl[ 1 - \eta_0\Bigl[20 \sqrt{\varepsilon} \Bigl(|\xi|-\sqrt{\frac{3}{5\varepsilon}}\Bigr) \Bigr] \Bigr]\widehat{f}
  $$
  and
  $$
  \widehat{P_{B_\varepsilon} f} = \Bigl[ 1 - \eta_0\Bigl[20 \sqrt{\varepsilon} \Bigl(|\xi|-\sqrt{\frac{3}{10\varepsilon}}\Bigr) \Bigr] \Bigr]\widehat{f}
  $$
  Clearly, $  \widehat{P_{A_\varepsilon} f}  $ cancels in a region of order $ \varepsilon^{-1/2}Ê$ around $\sqrt{\frac{3}{5 \varepsilon}}$ whereas
   $  \widehat{P_{B_\varepsilon} f}  $ cancels in a region of order $ \varepsilon^{-1/2}Ê$ around $\sqrt{\frac{3}{10 \varepsilon}}$.
  We are now in position to state the main proposition of this section :
  \begin{proposition}\label{pro1}
  Let  $ s\ge 1 $, $ 0<T<1 $ and $ u_{i,\varepsilon} \in C([0,T]; H^s(\R)) $, $i=1,2$,  be two solutions to $ (K_\varepsilon) $ with $ 0<\varepsilon<\! <1 $ and initial data $ \varphi_i $. Then it holds
  \begin{equation}\label{est1pro1}
  \| P_{A_\varepsilon} u_{i,\varepsilon} \|_{X^{s,1/2,1}_{\varepsilon,T} }\lesssim  \|\varphi_i\|_{H^s}  +T^{1/4}\| u_{i,\varepsilon}\|_{Y_{\varepsilon,T}^s} \| u_{i,\varepsilon}\|_{Y_{\varepsilon,T}^1}
  (1+ \| u_{i,\varepsilon}\|_{Y_{\varepsilon,T}^1})
  \end{equation}
  and, setting $ w=u_{1,\varepsilon} - u_{2,\varepsilon}$,
   \begin{equation}\label{est2pro1}
  \| P_{A_\varepsilon}w \|_{X^{s,1/2,1}_{\varepsilon,T} }\lesssim  \|\varphi_1-\varphi_2\|_{H^s}  +T^{1/4}\| w\|_{Y_{\varepsilon,T}^s}
  \sum_{i=1}^2\| u_{i,\varepsilon}\|_{Y_{\varepsilon,T}^s}
  (1+ \| u_{i,\varepsilon}\|_{Y_{\varepsilon,T}^s})
  \end{equation}

  where
  \begin{equation}\label{defY}
 \| u\|_{Y_{\varepsilon,T}^s}:=  \| P_{A_\varepsilon} u \|_{X^{s,1/2,1}_{\varepsilon,T}} +\|u\|_{L^\infty_T H^s}
  \end{equation}
  \end{proposition}
 We will make a frequent use of the following linear estimates
 \begin{lemma}\label{linear}
 Let $ \varphi\in S(\R) $ and  $ T\in ]0,1] $ then $\forall 0<\varepsilon \ll 1 $,
  \begin{eqnarray}
  \| P_{A_\varepsilon} \partial_x U_\varepsilon (t) \varphi \|_{L^\infty_x L^2_t}  & \lesssim  & \| \varphi\|_{L^2}  \label{linear1}\\
    \| D^{1/4}_x P_{B_\varepsilon} U_\varepsilon (t) \varphi \|_{L^4_t L^\infty_x}  & \lesssim  & \| \varphi\|_{L^2}   \label{linear2}\\
  \| P_{\le 2} U_\varepsilon (t) \varphi \|_{L^2_x L^\infty_T}  & \lesssim  & \| \varphi\|_{L^2}  \label{linear3}\; ,
  \end{eqnarray}
where $ {\mathcal F}_x (P_{\complement A_\varepsilon}\varphi)= (1-\eta_{A_\varepsilon}){\mathcal F}_x \varphi $ and  the implicit constants are independent of $ \varepsilon>0 $.
 \end{lemma}Ê
 \begin{proof}
First, \eqref{linear1} follows from the classical proof of the local Kato smoothing effect, by using that   $ |\phi_\varepsilon'(\xi)| \gtrsim |\xi|^2  $ on the Fourier support of $ P_{A_{\varepsilon}} $.

To prove \eqref{linear2},  we first notice that the Fourier support  of $  P_{B_{\varepsilon}} $
  does not intersect the region
  $ \{ \xi\in \R, \, |\xi|\in [\sqrt{\frac{1}{4\varepsilon}},\sqrt{\frac{7}{20\varepsilon}}] \}$. By the $ TT^{*} $ argument it suffices to prove that
\begin{equation}
\| U_\varepsilon (t) D^{1/2}_x P_{\complement B_\epsilon}  \varphi \|_{L^\infty_x} +\| U_\varepsilon (t) D^{1/2}_x P_{A_\epsilon}  \varphi \|_{L^\infty_x} \lesssim t^{-1/2}\| \varphi\|_{L^1}\label{to2}
\end{equation}
By classical arguments,  \eqref{to1} will be proven if we show
$$
\Bigl\| \int_{\R} \chi_{\{ |\xi|\not\in [\sqrt{\frac{1}{4\varepsilon}},\sqrt{\frac{7}{20\varepsilon}}]\}}
|\xi|^{1/2} e^{i[x\xi+(\xi^3-\varepsilon \xi^5)t]}\, d\xi\Bigr\|_{L^\infty_x} \lesssim t^{-1/2}  \; .
$$
Setting $ \theta:=\xi |t|^{1/3} $ this is equivalent to prove
\begin{equation}\label{too3}
I_{\epsilon}:=\sup_{t\in \R, X\in\R} \Bigl|\int_{\R}  \chi_{\{ |\theta|\not\in [\sqrt{\frac{|t|^{2/3}}{4\varepsilon}},\sqrt{\frac{7|t|^{2/3}}{20\varepsilon}}]\}} \, |\theta|^{1/2} \, e^{i[X\theta+\theta^3-\frac{\varepsilon}{|t|^{2/3}} \theta^5]}\, d\theta \Bigr| \lesssim 1
\end{equation}
We set $ \Phi(\theta)=\Phi_{t,\varepsilon}(\theta):=\theta^3-\frac{\varepsilon}{|t|^{2/3}} \theta^5$ and notice that
$$
 \Phi'(\theta):=
3 \theta^2 - \frac{5\varepsilon}{|t|^{2/3}} \theta^4 \;\mbox{ and  }\;
 \Phi^{''}(\theta)=2\theta \Bigl( 3   - \frac{10\varepsilon}{|t|^{2/3}} \theta^2 \Bigr)\; .$$
\eqref{too3} is obvious when restricted on $ |\theta|\le 100$.  Now, it is worth noticing that
$$
|\Phi''(\theta)|\gtrsim 1+\max\Bigl(|\theta|,  \frac{\varepsilon}{|t|^{2/3}} \theta^3 \Bigr)
$$
whenever  $\theta\in \{ |z|\ge 100\, / \, |z|\not\in [\sqrt{\frac{|t|^{2/3}}{4\varepsilon}},\sqrt{\frac{7|t|^{2/3}}{20\varepsilon}}]\} $.
 Therefore, in the region $ |\theta| \in [\sqrt{\frac{|t|^{2/3}}{10\varepsilon}},\sqrt{\frac{2|t|^{2/3}}{\varepsilon}}]$,
 \eqref{too3} follows from Van der Corput lemma since
 $ |\Phi''(\theta)|\gtrsim 1+\frac{|t|^{1/3}}{\sqrt{\varepsilon}}  $ and $ |\theta|^{1/2}
  \sim \frac{|t|^{1/6}}{\varepsilon^{1/4}}$.  It thus remains to consider the region
 $ |\theta|\not\in  [\sqrt{\frac{|t|^{2/3}}{10\varepsilon}},\sqrt{\frac{2|t|^{2/3}}{\varepsilon}}]$.
 We notice that, in this region, it holds
 \begin{equation}\label{qs}
 |\Phi'(\theta)|\sim |\theta|^2 \mbox{ for } |\theta|\le \sqrt{\frac{|t|^{2/3}}{10\varepsilon}} \; \mbox{ and }\;
  |\Phi'(\theta)|\sim \frac{\varepsilon  |\theta|^4}{ |t|^{2/3}}\mbox{ for } |\theta|\ge\sqrt{\frac{2|t|^{2/3}}{\varepsilon}} \;
\end{equation}
and divide this region into two subregions.\\
$  \bullet $ The subregion $|\Phi'(\theta)-X|\le |X|/ 2 $. Then $ |\Phi'(\theta)|\sim |X| $. Assuming we are in the region $ 100<|\theta|\le  \sqrt{\frac{|t|^{2/3}}{10\varepsilon}} $, we have $ |\Phi'(\theta)|\sim |\theta|^2 $ and thus  $|\theta|\sim \sqrt{|X|}Ê$. Then  \eqref{too3} follows from Van der Corput lemma
 since $ |\Phi''(\theta)|\gtrsim |\theta| \sim  \sqrt{|X|}$.
  On the other hand, assuming that $ |\theta|\ge\sqrt{\frac{2|t|^{2/3}}{\varepsilon}}\ge 100 $ then
   $  |\Phi'(\theta)|\sim \varepsilon  |\theta|^4|t|^{-2/3} $ and  thus
    $ |\theta|\sim \varepsilon^{-1/4}|X|^{1/4} |t|^{1/6}Ê$. \eqref{too3}  follows again from
   Van der Corput lemma
 since $ |\Phi''(\theta)|\gtrsim |\theta|\sim  \varepsilon^{-1/4}|X|^{1/4} |t|^{1/6}Ê$.\\
   $  \bullet $ The subregion $|\Phi'(\theta)-X| > |X|/ 2 $. Then $ |\Phi'(\theta)-X|\sim |\Phi'(\theta) | $ 
   and \eqref{too3} is obtained by integrating by parts and using \eqref{qs}.  This completes the proof of \eqref{linear2}.

Finally, to show \eqref{linear3} we  notice that it suffices to prove that for $ |x|\ge 10^{4} $,
$$
\sup_{t\in [0,1] } \Bigl| \int_{\R}  \eta_{\le 2}(\xi) \, e^{i[x\xi+\phi_{\varepsilon}(\xi)t]}\, d\xi\Bigr|
\lesssim |x|^{-2}\; ,
$$
where $ \phi_\varepsilon(\xi)=\xi^3-\varepsilon \xi^5$. 
But this follows directly by integrating by parts twice since $ |x-\phi_\varepsilon'(\xi)t |\gtrsim |x| $ for any
 $ |t|\le 1 $ and $| \xi|\le 4 $.
 \end{proof}
 To prove Proposition \ref{pro1} we will have to put the whole solution $ u_\varepsilon $ of $(K_\varepsilon) $ and not only $ P_{A_\varepsilon} u_\varepsilon$ in some Bourgain's space with regularity 1 in time. This will be done in the next lemma by noticing that any solution to $(K_\varepsilon) $ that belongs to $ C([0,T]; H^1(\R)) $ automatically belongs to
  $X^{0,1}_{\varepsilon,T} $.
  \begin{lemma}\label{oo}
  Let  $ s\ge 1 $, $ T\in ]0,1[  $ and $ u \in C([0,T];H^s(\R)) $ be a solution to $(K_\varepsilon) $. Then,
  \begin{equation}\label{ze}
  \|u\|_{X^{s-1,1}_{\varepsilon,T}} \lesssim  \|u\|_{L^\infty_T H^{s-1}_x} + \|u \|_{L^\infty_T H^1_x}
    \|u \|_{L^\infty_T H^{s}_x} \; ,
  \end{equation}
   where the implicit constant is independent of $ \varepsilon$.
  \end{lemma}
  \begin{proof}
  First, we consider $v(t)=U_\varepsilon (-t)u(t)$ on the time interval $]0,T]$ and extend
$v$ on $]-2,2[$ by setting $\partial_tv=0$ on $[-2,2] \setminus [0,T]$.
Then, it is pretty clear that
\begin{displaymath}
\|\partial_tv\|_{L^2(]-2,2[;H^{s-1}_x)}=\|\partial_tv\|_{L^2_T H^{s-1}_x},
\quad \text{and} \quad \|v\|_{L^2(]-2,2[;H^{s-1}_x)} \lesssim
 \|v\|_{L^{\infty}_{T} H^{s-1}_x} \; .
\end{displaymath}
Now, we define
$\tilde{u}(x,t)=\eta(t)U(t)v(t)$. Obviously, $ \tilde{u} $ is an extension of $ u$ outside $ ]-T,T[ $ and  it holds
\begin{equation} \label{apriori u.1b}
\|\tilde{u}\|_{X_\varepsilon^{s-1,1}}\lesssim
\|\partial_tv\|_{L^2(]-2,2[;H^{s-1}_x) }
+\|v\|_{L^2(]-2,2[;H^{s-1}_x)}
\lesssim \|\partial_tv\|_{L^2_T H^{s-1}_x}
+ \|v\|_{L^{\infty}_{T} H^{s-1}_x} \; .
\end{equation}
Therefore  \eqref{ze} follows from the identity
\begin{displaymath}
\partial_tv=U_\varepsilon (-t) \Bigl[ u_t +u_{xxx}Ê+\varepsilon u_{5x}\Big] 
\end{displaymath}
together with the facts that $u$ is a solution to $(K_\varepsilon) $ and that 
$$
\|u u_x \|_{H^{s-1}_x} \le \| u^2 \|_{H^s_x} \lesssim \| u\|_{L^\infty_x}  \| u \|_{H^s_x}
$$
as soon as $ s\ge 1$.
  \end{proof}
   Now, according to the Duhamel formula and to classical linear estimates in Bourgain's spaces (cf. \cite{Bourgain1993}, \cite{Gi}), Proposition \ref{pro1} is a direct consequence of the following bilinear estimate
 \begin{eqnarray}
 \| P_{A_\varepsilon} \partial_x (u_1 u_2 ) \|_{X^{s,-1/2,1}_\varepsilon}
  &\lesssim & 
  T^{1/4}\Bigl(  \|u_1 \|_{Y^s_\varepsilon}+  \|u_1 \|_{X^{s-1,1}_\varepsilon}\Bigr)
  \Bigl(  \|u_2 \|_{Y^1_\varepsilon}+  \|u_2 \|_{X^{0,1}_\varepsilon}\Bigr)\nonumber \\
   & & +   T^{1/4} \Bigl(  \|u_1 \|_{Y^1_\varepsilon}+  \|u_1 \|_{X^{0,1}_\varepsilon}\Bigr)
  \Bigl(  \|u_2 \|_{Y^{s-1}_\varepsilon}+  \|u_2 \|_{X^{s-1,1}_\varepsilon}\Bigr)
  \label{bilinear} \; ,
 \end{eqnarray}
 where the functions $ u_i $ are supported in time in $ ]-T,T[ $ with $ 0<T\le 1 $. 
To prove this bilinear estimate we first note  that by symmetry it
suffices to consider $ \partial_x \Lambda(u,v) $ where
$\Lambda(\cdot,\cdot) $ is defined by
$$
{\mathcal F}_x (\Lambda(u,v)):=\int_{\R} \chi_{|\xi_1|\le
|\xi-\xi_1|} ({\mathcal F}_x u) (\xi_1) ({\mathcal F}_x v)
(\xi-\xi_1)\, d\xi_1\; .
$$
Moreover, using that for any $ s\ge 1 $, 
$$
\langle \xi_1+\xi_2\rangle^s \lesssim 
\langle \xi_1+\xi_2\rangle\Bigl( \langle \xi_1\rangle^{s-1}+\langle \xi_2\rangle^{s-1}\Bigr) \; ,
$$
it is a classical fact that we can restrict ourself to prove  \eqref{bilinear} for $ s=1 $.\\
As mentioned in the introduction, the  following resonance relation
is crucial for our analysis in this frequency area :
\begin{equation}\label{resonance}
\Theta(\xi,\xi_1):=\sigma-\sigma_1-\sigma_2= \xi \xi_1 (\xi-\xi_1) \Bigr[3-5 \varepsilon \Bigl( (\xi_1+\xi_2)^2-\xi_1\xi_2\Bigr) \Bigr]
\end{equation}
where
$$
\sigma:=\sigma(\tau,\xi):=\tau-\xi^3-\varepsilon \xi^5 ,\quad
 \sigma_1:=\sigma(\tau_1,\xi_1) \;
\mbox{ and } \sigma_2:=\sigma(\tau-\tau_1,\xi-\xi_1)\; .
$$
We start by noticing that the case of ouput frequencies of order less or equal to one is harmless. Indeed, it is easy to check that for any couple 
$ u_i $, $ i=1,2 $,  of  smooth functions supported in time in $ ]-T,T[ $ with $ 0<T\le 1 $ it holds 
\begin{equation}\label{eqeq1}
\|\partial_x P_{A_\varepsilon} P_{\le 8} \Lambda( u_1,u_2) \|_{X^{1,-1/2,1}_\varepsilon }\lesssim \|  \Lambda( u_1,u_2)\|_{L^2} 
\lesssim \|u_1\|_{L^\infty_t H^1} \|u_2 \|_{L^\infty_t H^1} 
  \; .
 \end{equation}
Let us continue by deriving  an estimate for  the interactions of high frequencies with frequencies of order less or equal to  1.
\begin{lemma}\label{lemme1} Let $ u_i $, $ i=1,2 $, be two smooth functions supported in time in $ ]-T,T[ $ with $ 0<T\le 1 $.  Then it holds 
\begin{equation}\label{eqeq}
\|\partial_x P_{A_\varepsilon}\Lambda(P_{\le 8} u_1,u_2) \|_{X^{1,-1/2,1}_\varepsilon }
\lesssim  \|u_1\|_{X^{0,1}_\varepsilon}
 \Bigl( T^{1/4}(\|P_{A_\varepsilon} u_2\|_{X^{1,1/2,1}_\varepsilon}+ \| u_2\|_{X^{0,1}_\varepsilon})
 +\| \partial_x u_2\|_{L^2_{tx}} \Bigr)
  \; .
 \end{equation}
  \end{lemma}
\begin{proof}  Since the norms in the right-hand side of \eqref{eqeq} only see the size of the modulus of the Fourier transform, we can assume that all our functions have non negative Fourier transform.
We  set $ \eta_{A_\varepsilon} =1 - \eta_0\Bigl[20 \sqrt{\varepsilon}
\Bigl(|\xi|-\sqrt{\frac{3}{5\varepsilon}}\Bigr) \Bigr] $ so that
 $ \widehat{P_{A_\varepsilon} f} = \eta_{A_\varepsilon} \widehat{f} $.
Rewriting $ \eta_{A_\varepsilon}(\xi) $ as  $\eta_{A_\varepsilon}(\xi-\xi_1) +( \eta_{A_\varepsilon}(\xi) - \eta_{A_\varepsilon}(\xi-\xi_1) )$, it suffices to estimate the two following terms
$$
I_1:=\Bigl\| {\mathcal F}^{-1}_{x} \Bigl( \partial_x \Lambda(\eta_{\le 8} {\mathcal F}_x (u_1) , \eta_{A_\varepsilon}
 {\mathcal F}_x (u_2) \Bigr) \Bigr\|_{X^{1,-1/2,1}}$$
 and
 $$
I_2:=\Bigl\| {\mathcal F}^{-1}_{x}\Bigl(  \xi \, \int_{\R} \eta_{\le 8}(\xi_1) {\mathcal F}_x (u_1)(\xi_1)  (\eta_{A_\varepsilon}(\xi)-\eta_{A_\varepsilon}(\xi-\xi_1))
 {\mathcal F}_x (u_2) (\xi-\xi_1)\, d\xi_1\Bigr) \Bigr\|_{X^{1,-1/2,1}}$$
 $ I_1$ is easily estimate thanks to \eqref{linear3} by
 \begin{eqnarray*}
 I_1^2 & \lesssim &\sum_{N\ge 1} T^{\frac{1}{2}-}\Bigl\| (\eta_{\le 8} \widehat{u_1})\ast (\eta_{N} \eta_{A_\varepsilon}
  \widehat{\partial_x^2 u_2})\Bigr\|_{L^2}^2 \\
  & \lesssim & T^{\frac{1}{2}-} \sum_{N\ge 1} \| P_{\le 8} u_1 \|_{L^2_x L^\infty_t}^2 \|\partial_x^2 P_N P_{A_\varepsilon} u_2
   \|_{L^\infty_x L^2_t}^2 \\
    & \lesssim & T^{\frac{1}{2}-} \|u_1\|_{X^{0,1}}^2 \| P_{A_\varepsilon} u_2 \|_{X^{1,1/2,1}}^2\quad .
 \end{eqnarray*}
To estimate $ I_2 $ we first notice that for $ |\xi_1|\le 4 $ and $ 0<\varepsilon<10^{-8}Ê$,
  \begin{equation}\label{relare1}
 \eta_{A_\varepsilon}(\xi)-\eta_{A_\varepsilon}(\xi-\xi_1)=0 \mbox{ whenever } |\xi|  \in \Bigr[\frac{15}{16} \sqrt{\frac{3}{5\varepsilon}},
\frac{17}{16} \sqrt{\frac{3}{5\varepsilon}} \Bigl]\cup \complement \Bigl[\frac{2^{-3}}{\sqrt{\varepsilon}}, \frac{2^3}{\sqrt{\varepsilon}}\Bigr] \; .
   \end{equation}
and for any $(\xi,\xi_1)  \in \R^2 $,
 \begin{equation}\label{relare2}
 |\eta_{A_\varepsilon}(\xi)-\eta_{A_\varepsilon}(\xi-\xi_1)| \lesssim \min \Bigl( 1, \sqrt{\varepsilon}
 |\xi_1| \Bigr) \; .
  \end{equation}
Moreover,  in the region
 $ |\xi_1|\le 4 $ and   $|\xi|  \not \in [\frac{15}{16}  \sqrt{\frac{3}{5\varepsilon}},
\frac{17}{16}  \sqrt{\frac{3}{5\varepsilon}}] $ the resonance relation \eqref{resonance}Ê ensures that
 \begin{equation}\label{relare3}
 |\sigma_{max}|:=\max(|\sigma|,|\sigma_1|,|\sigma_2|)\gtrsim |\xi \xi_1 (\xi-\xi_1)|
 \end{equation}
 where $ \sigma(\tau,\xi):=\tau-\phi_{\varepsilon}(\xi)$, $\sigma_1=\sigma(\tau_1,\xi_1)$ and
  $ \sigma_2=\sigma(\tau-\tau_1,\xi-\xi_1)$.
  We separate three regions  \\
  $\bullet $ $\sigma_{max}=\sigma_2 $.  Then according to \eqref{relare1}-\eqref{relare3},
  \begin{eqnarray*}
  I_2 & \lesssim & T^{\frac{1}{2}-} \Bigl\|\int_{\R^2} (\eta_{\le 8} \widehat{u_1})(\xi_1,\tau_1) \sqrt{\varepsilon}
  \frac{|\xi_1|\langle \xi\rangle^2}{|\xi_1||\xi-\xi_1|^2} \langle \sigma_2 \rangle \chi_{\{|\xi-\xi_1|\sim  \frac{1}{\sqrt{\varepsilon}}\}}
  \widehat{u_2}(\xi-\xi_1,\tau-\tau_1)\, d\xi_1\, d\tau_1 \Bigr\|_{L^2_{\xi,\tau}(|\xi|\sim \frac{1}{\sqrt{\varepsilon}})} \\
  & \lesssim & T^{\frac{1}{2}-} \| P_{\le 8}u_1 \|_{L^\infty_{tx}} \| u_2 \|_{X^{-1/2,1}} \\
  & \lesssim & T^{\frac{1}{2}-} \| u_1 \|_{X^{0,1}} \| u_2 \|_{X^{0,1}}
  \end{eqnarray*}
   $\bullet $ $\sigma_{max}=\sigma_1 $.  Then according to \eqref{relare1}-\eqref{relare3},
  \begin{eqnarray*}
  I_2 & \lesssim & T^{\frac{1}{2}-} \Bigl\| \langle \xi \rangle^2 \int_{\R^2} \frac{\langle \sigma_1 \rangle}{|\xi_1| |\xi-\xi_1|^2}  (\eta_{\le 8} \widehat{u_1})(\xi_1,\tau_1) \sqrt{\varepsilon} |\xi_1|
 \chi_{\{|\xi-\xi_1|\sim  \frac{1}{\sqrt{\varepsilon}}\}}
  \widehat{u_2}(\xi-\xi_1,\tau-\tau_1)\, d\xi_1\, d\tau_1 \Bigr\|_{L^2_{\xi,\tau}(|\xi|\sim \frac{1}{\sqrt{\varepsilon}})} \\
  & \lesssim &T^{\frac{1}{2}-}  \| u_1 \|_{X^{0,1}} \| D_x^{-1/2} {\mathcal F}^{-1}(\chi_{\{|\xi|\sim
    \frac{1}{\sqrt{\varepsilon}}\}}\widehat{u_2}) \|_{L^\infty_{tx}} \\
  & \lesssim & T^{\frac{1}{2}-} \| u_1 \|_{X^{0,1}} \|  {\mathcal F}^{-1}(\chi_{\{|\xi|\sim
    \frac{1}{\sqrt{\varepsilon}}\}}\widehat{u_2} \|_{L^\infty_{t}L^2_x} \\
    & \lesssim & T^{\frac{1}{2}-} \| u_1 \|_{X^{0,1}} \| u_2 \|_{X^{0,1}}
  \end{eqnarray*}
   $\bullet $ $\sigma_{max}=\sigma $.  Then according to \eqref{relare1}-\eqref{relare3},
  \begin{eqnarray*}
  I_2 & \lesssim & \Bigl\| \langle \xi \rangle^2 \int_{\R} \frac{\sqrt{\varepsilon} |\xi_1|}{|\xi_1|^{3/8} |\xi-\xi_1|^{3/4}}
     (\eta_{\le 8} \widehat{u_1})(\xi_1)
 \chi_{\{|\xi-\xi_1|\sim  \frac{1}{\sqrt{\varepsilon}}\}}
  \widehat{u_2}(\xi-\xi_1)\, d\xi_1 \Bigr\|_{L^2(|\xi|\sim \frac{1}{\sqrt{\varepsilon}})} \\
  & \lesssim & \sqrt{\varepsilon} \| P_{\le 8} u_1 \|_{L^\infty_{tx}}
   \| D_x^{5/4} {\mathcal F}^{-1}(\chi_{\{|\xi|\sim
    \frac{1}{\sqrt{\varepsilon}}\}}\widehat{u_2}) \|_{L^2_{tx}} \\
  & \lesssim & \| u_1 \|_{X^{0,1}} \| \partial_x u_2 \|_{L^2_{tx}}
  \end{eqnarray*}
  This completes the proof of the lemma.
  \end{proof}
The next lemma ensures that the restriction of the left-side member of \eqref{bilinear} on the region $ |\xi|\gtrsim 1 $, $ |\xi_1|\gtrsim 1 $ and $ |\sigma_{max}| \ge 2^{-5} |\xi \xi_1 (\xi-\xi_1) | $ can be easily
 controlled.
 \begin{lemma}\label{lemme2}
Under the same hypotheses as in Lemma \ref{lemme1},  in the region  where the following strong resonance relation holds 
 \begin{equation}\label{strong}
  |\sigma_{max}| \ge 2^{-5} |\xi \xi_1 (\xi-\xi_1) |\; ,
  \end{equation}
    we have
 \begin{equation}\label{eqlemme2}
 \|\partial_x ÊP_{A_\varepsilon} P_{\ge 8}\Lambda(P_{\ge 8} u_1,u_2) \|_{X^{1,-1/2,1}_\varepsilon }  \lesssim 
 T^{1/4}  \|u_1\|_{X^{0,1}}  \|u_2\|_{X^{0,1}} +
  \Bigl( \|u_1\|_{X^{0,1}} + \|\partial_x u_1\|_{L^2_{tx}} \Bigr)   \|\partial_x u_2\|_{L^2_{tx}}\; .
 \end{equation}
 \end{lemma}
\begin{proof} Again we notice that the norms in the right-hand side of \eqref{lemme2}   only see the size of the modulus of the Fourier transforms. We can thus assume that all our functions have non-negative Fourier transforms.
 We set $ I:=  \|\partial_x P_{A_\varepsilon}P_{\ge 8}\Lambda(P_{\ge 8} u_1,u_2) \|_{X^{1,-1/2,1}_\varepsilon }  $  and separate different subregions .\\
  $\bullet $ $|\sigma_1|\ge 2^{-5}  |\xi \xi_1 (\xi-\xi_1) |$. Then direct calculations  give
  \begin{eqnarray*}
  I & \lesssim  & T^{\frac{1}{2}-}  \| u_1 \|_{X^{0,1}} \| D_x^{-1} P_{\ge 2}u_2 \|_{L^\infty_{tx}}\\
  & \lesssim &  T^{\frac{1}{2}-}  \| u_1 \|_{X^{0,1}} \| u_2 \|_{X^{0,1}} \; .
  \end{eqnarray*}
   $\bullet $ $|\sigma_2|\ge 2^{-5}  |\xi \xi_1 (\xi-\xi_1) |$. This case can be treated exactly in the same way by exchanging the role of
   $u_1 $ and $ u_2 $.\\
     $\bullet $ $|\sigma|\ge 2^{-5}  |\xi \xi_1 (\xi-\xi_1) | $ and $ \max (|\sigma_1|, \sigma_2|) \le  2^{-5}  |\xi \xi_1 (\xi-\xi_1) |$. \\
     Then we separate two subregions. \\
   1. $ |\xi_1|\ge 2^{-7} |\xi| $. Then $ |\xi_1|Ê\gtrsim |\xi_{max}| $ and taking $ \delta>0 $ close enough to $ 0 $ we get
    \begin{eqnarray*}
  I & \lesssim  & \|\partial_x P_{A_\varepsilon}P_{\ge 8}\Lambda(P_{\ge 8} u_1,u_2) \|_{X^{1,-1/2+\delta}_\varepsilon } \\
  & \lesssim &  \Bigl\| \partial_x u_2\, D_x^{-1/2+3\delta}P_{\ge 8} u_1  \Bigr\|_{L^2} \\
  & \lesssim &  \|D_x^{-1/2+3\delta}P_{\ge 8} u_1\|_{L^\infty_{tx}} \|\partial_x u_2 \|_{L^2_{tx}} \\
  & \lesssim &     \|u_1\|_{X^{1/4,3/4}}  \|\partial_x u_2 \|_{L^2_{tx}} \\
    & \lesssim &    ( \|u_1\|_{X^{0,1}} + \|\partial_x u_1 \|_{L^2_{tx}} )  \|\partial_x u_2 \|_{L^2_{tx}} \; .
  \end{eqnarray*}
2. $ |\xi_1|\le 2^{-7} |\xi| $.  Then, we  notice that  in this region $\frac{1}{2} |\xi|\le |\xi-\xi_1| \le 2  |\xi|Ê$ and thus
$$
(1-2^{-6}) \xi^2 \le \xi^2-\xi_1 (\xi-\xi_1) \le (1+2^{-6})\xi^2 \; .
$$
Since $ \eta_{A_\varepsilon} $ does vanish on   $ \Bigl\{|\xi|\in \Bigr[ \frac{15}{16} \sqrt{\frac{3}{5\varepsilon}},\frac{17}{16} \sqrt{\frac{3}{5\varepsilon}} \Bigr]\Bigr\} $, we deduce from \eqref{resonance} that
\begin{equation} \label{vz}
|\sigma| \sim \max\Bigl( |\xi \xi_1 (\xi-\xi_1) |, \varepsilon |\xi^3 \xi_1 (\xi-\xi_1) |\Bigr)
\end{equation} 
on the support of $  \eta_{A_\varepsilon} $. We thus can write
   \begin{eqnarray*}
  I^2 & \lesssim & \sum_{N\ge 4} \Bigl( \sum_{4\le N_1\le 2^{-5}N}
   \Bigl\| \eta_{N}(\xi) \eta_{A_\varepsilon}(\xi)  |\xi|Ê \chi_{\{|\sigma|\sim  \max(N_1 N^2 ,  \varepsilon N^4 N_1 ) \} }{\mathcal F}_x \Bigl(
    \Lambda(P_{N_1} u,  u_2)\Bigr) \Bigr\|_{X^{1,-1/2,1}_\varepsilon }\Bigr)^2 \\
     & \lesssim  &  \sum_{N\ge 4} \Bigl( \sum_{4\le N_1\le 2^{-5}N} \| P_{N_1} D_x^{-1/2}u_1 \|_{L^\infty_{tx}} \|\chi_{\{|\xi|\sim N\}} \xi  \,\widehat{u_2}\|_{L^2_{ \tau,\xi}}\Bigr)^2 \\
& \lesssim  &    \sum_{N\ge 4}   \|\chi_{\{|\xi|\sim N\}} \xi  \,\widehat{u_2}\|_{L^2_{ \tau,\xi}}^2 \Bigl( \sum_{4\le N_1\le 2^{-5}N} N_1^{-1/4} \| P_{N_1} D_x^{1/4}u_1 \|_{L^\infty_{t}L^2_x}   \Bigr)^2 \\
  & \lesssim &     \|u_1\|_{X^{1/4,3/4}}^2  \|\partial_x u_2 \|_{L^2_{tx}}^2 \\
    & \lesssim &    ( \|u_1\|_{X^{0,1}} + \|\partial_x u_1 \|_{L^2_{tx}} )^2  \|\partial_x u_2 \|_{L^2_{tx}}^2 \; .
  \end{eqnarray*}
  \end{proof}
  \noindent
{\bf Proof of the bilinear estimate \eqref{bilinear}} \\
First, according to \eqref{eqeq1} and Lemma \ref{lemme1}Ê and  to the support of $ \eta_{A_\varepsilon} $ it suffices to consider
$$
I:=\Bigr[\sum_{N\ge 4 } N^2 \Bigl( \sum_{L} L^{-1/2}\Big\| \eta_{L}(\sigma) \eta_{N}(\xi) \int_{\R^2} \sum_{N_1\wedge N_2\ge 8} \widehat{P_{N_1} u_1}Ê(\xi_1,\tau_1)
 \widehat{P_{N_2} u_2}(\xi_2,\tau_2) \, d\tau_1 \, d\xi_1 \Bigr\|_{L^2_{\tau,\xi} (|\xi|\not \in J_\varepsilon)}\Bigr)^{2}\Bigr]^{1/2}
\; ,
 $$
 where  
 \begin{equation}\label{defJ}
 J_\varepsilon=[ \frac{15}{16} \sqrt{\frac{3}{5\varepsilon}},\frac{17}{16} \sqrt{\frac{3}{5\varepsilon}} ],
 \quad  \tau_2= \tau-\tau_1 \; \mbox{ and }\;  \xi_2=\xi-\xi_1 \; .
 \end{equation}
Now we will decompose the region of integration Ê into different regions and we will check that in most of these regions the strong resonance relation
 \eqref{strong} holds.  By symmetry we can assume that $ N_1 \le N_2 $. For the remaining it is convenient to introduce the function
  $$
  \Gamma(\xi,\xi_1) :=\Bigl|3-5 \varepsilon \Bigl( \xi^2-\xi_1(\xi-\xi_1)\Bigr)\Bigr|
  $$
   which is  related to the resonance relation \eqref{resonance}. 
    \begin{enumerate}
 \item[1.]    $ N_1< 2^{-10} N_2 $. Then  it holds 
 $$
 (1-2^{-7})\xi^2 \le \xi^2 -\xi_1(\xi-\xi_1) \le (1+2^{-7})\xi^2 \
 $$
 and it  is  easy to check that $  \Gamma(\xi,\xi_1) \ge 2^{-5} $ as soon as $ |\xi|\not \in  J_\varepsilon $. According to \eqref{resonance} this ensures that \eqref{strong} holds.
 \item[2.]  $ N_1\ge  2^{-10} N_2 $.
 \begin{enumerate}
\item[2.1.]  The subregion $|\xi|\not\in  \Bigl[ \sqrt{\frac{17}{80 \varepsilon}}, \sqrt{\frac{2}{5 \varepsilon}} \Bigr] $. In this region, by \eqref{linear2}Ê of Lemma \ref{linear} and duality,
  we get
 \begin{eqnarray*}
 I & \lesssim & \sum_{\min(4,Ê2^{-10}N_2)<N_1\le N_2}\| D_x^{-\frac{1}{4}+} \partial_x^2( P_{N_1} u_1 P_{N_2} u_2) \|_{L^{\frac{4}{3}+}_t L^{1+}_x} \\
 & \lesssim &  \sum_{\min(4,Ê2^{-10}N_2)<N_1\le N_2} T^{\frac{3}{4}-}
  N_2^{-\frac{1}{4}+}\| \partial_x P_{N_1} u_1 \|_{L^\infty_t L^2_x}  \|   \partial_x P_{N_2} u_2 \|_{L^\infty_t L^{2+}_x}\\
 & \lesssim & T^{\frac{3}{4}-} \|u_1 \|_{L^\infty_t H^1}\|u_2 \|_{L^\infty_t H^1}\; .
 \end{eqnarray*}
 \item[2.2.] The subregion $|\xi|\in  \Bigl[ \sqrt{\frac{17}{80 \varepsilon}}, \sqrt{\frac{2}{5 \varepsilon}} \Bigr] $.
 \begin{enumerate}
  \item[2.2.1] The subregion  $ |\xi_1|\wedge |\xi_2| \le \sqrt{\frac{17}{80 \varepsilon}} $. Since both cases can be treated in the same way, 
  we assume $  |\xi_1|\wedge |\xi_2|=|\xi_1| $. Then, according to  \eqref{linear2}  and the support of $ \eta_{A_\varepsilon} $ and  Ê$ \eta_{B_\varepsilon} $, we get
  \begin{eqnarray*}
 I & \lesssim &  \sum_{\min(4,Ê2^{-10}N_2)<N_1\le N_2}
 T^{\frac{1}{2}-}\|  \partial_x^2( P_{B_\varepsilon}P_{A_\varepsilon}  P_{N_1} u_1 P_{N_2}u_2) \|_{L^2_{tx}} \\
 & \lesssim &  T^{\frac{1}{2}-} \sum_{\min(4,Ê2^{-10}N_2)<N_1\le N_2} 
\| P_{B_\varepsilon}P_{A_\varepsilon}   \partial_x P_{N_1} u_1 \|_{L^4_t L^\infty_x}  \|  \partial_x P_{N_2} u_2 \|_{L^\infty_t L^{2}_x}\\
  & \lesssim &   T^{\frac{1}{2}-}\sum_{\min(4,Ê2^{-10}N_2)<N_1\le N_2} N_1^{-1/4}\|P_{A_\varepsilon}    P_{N_1} u_1 \|_{X^{1,1/2,1 }}  \|  \partial_x P_{N_2} u_2 \|_{L^\infty_t L^{2}_x}\\
 & \lesssim & T^{\frac{1}{2}-} \|P_{A_\varepsilon} u_1 \|_{X^{1,1/2,1 }}\|u_2 \|_{L^\infty_t H^1}\; .
 \end{eqnarray*}
\item[2.2.2] The subregion $ |\xi_1|\wedge |\xi_2| > \sqrt{\frac{17}{80 \varepsilon}} $.  In this subregion we claim that 
 \eqref{strong} holds. 
Indeed, on one hand, if    $ \xi_1 \xi_2 \ge 0 $ then   $ \xi^2 - \xi_1 \xi_2 \le \xi^2 \le \frac{2}{5\varepsilon} $ and thus
 $  \Gamma(\xi,\xi_1)\ge 1  $. On the other hand, if  $ \xi_1 \xi_2 \le 0 $ then, since  $| \xi| \ge \sqrt{\frac{17}{80 \varepsilon}}$,  we must have
 $|\xi_1|\vee|\xi_2| \ge 2  \sqrt{\frac{17}{80 \varepsilon}} $. Therefore, $ \xi^2-\xi_1\xi_2 \ge 3 \frac{17}{80 \varepsilon} $ and thus
   $\Gamma(\xi,\xi_1)\ge \frac{3}{16} $ which ensures that \eqref{strong} holds and completes the proof of \eqref{bilinear}.
 \end{enumerate}
 \end{enumerate}
 \end{enumerate}
 \section{Uniform estimate close to the stationary point of the phase function}\label{section3}
 As announced in the introduction, close the the stationary point of the phase function we will apply the approach developed by Koch and Tzvetkov in \cite{KT1}. Note that, in  \cite{KK},Ê Kenig and Koenig improved this approach by adding the use of the nonlinear local Kato smoothing effect. However, this improvement can not be used here  since this smoothing effect is not uniform in $ \varepsilon $ close to  the stationary point.\\
 \begin{proposition}\label{propo2}
 Let   $ s\ge 1 $ and $ u_{\varepsilon} \in C([0,T]; H^s(\R)) $, $i=1,2$,  be a  solution to $(K_\varepsilon) $  with initial data $ \varphi $. Then it holds
 \begin{equation}\label{est1propo2}
 \|\|P_{\complement A_\varepsilon} u_{\varepsilon}  \|_{L^\infty_T H^s_x}^2 \lesssim \|P_{\complement A_\varepsilon} \varphi\|_{H^s}^2 +
 (\varepsilon^{1/2}+T^{1/4})
  \|u_{\varepsilon} \|_{Y_{\varepsilon,T}^s}^2 \Bigl(  \|u_{\varepsilon} \|_{Y_{\varepsilon,T}^1} +  \|u_{\varepsilon} \|_{Y_{\varepsilon,T}^1}^2)
 \end{equation}
  where $ Y_{\varepsilon,T}^sÊ$ is defined in \eqref{defY} and 
 $ {\mathcal F}_x (P_{\complement A_\varepsilon}\varphi)= (1-\eta_{A_\varepsilon}){\mathcal F}_x \varphi $.
 \end{proposition}
 First we establish an  estimate,  uniform in $ \varepsilon $,  on the solution to the associated non homogenous linear problem.
 \begin{lemma}\label{carlos}
Let $ v\in C([0,T];H^{\infty}(\R))$ be a solution of
\begin{equation}\label{Eqnonhom}
v_t+v_{xxx} +\varepsilon v_{5x}  = -F_x \quad .
\end{equation}
Then
\begin{equation}\label{estim1}
\| P_{\complement A_\varepsilon} v\|_{L^{1}_T L^\infty_{x}} \lesssim (\varepsilon^{1/2}+T)
\|P_{\complement A_\varepsilon} \,v\|_{L^\infty_T\, L^2_{x}}
+\|P_{\complement A_\varepsilon} F\|_{L^1_{T}L^2_{x}}\quad .
\end{equation}
\end{lemma}
\begin{proof}
For $ 0<\varepsilon<<1  $ fixed, we write a natural splitting
$$
[0,T] =\cup I_j
$$
of $ [0,T]Ê$ where $I_j=[a_j,b_j]$ are with disjoint interiors and $ |I_j|\leq
\varepsilon^{1/2}$.
Clearly, we can suppose that the number of the intervals $I_j$ is bounded by 
$C(1+T \varepsilon^{-1/2})$.
Using the H\"older inequality in time, we can write
\begin{equation*}
\| v \|_{L^{1}_T L^\infty_{x}} \lesssim  \sum_j \|v\|_{L^{1}_{I_j} L^\infty_{x}}
\lesssim \varepsilon^{\frac{3}{8}}\,
  \sum_j \|v \|_{L^{4}_{I_j} L^\infty_{x}}\;  .
\end{equation*}
Next, we apply the Duhamel formula on each $ I_j $ to obtain 
$$
P_{\complement A_\varepsilon} v(t)= U_\varepsilon (t-a_j)  P_{\complement A_\varepsilon} v(a_j) -
 \int_{a_j}^t  U_\varepsilon (t-t')   P_{\complement A_\varepsilon}
\partial_x F(t') \, dt' \; .
$$
Using the uniform in $ \varepsilon $ Strichartz estimate \eqref{linear2}  and classical $ T T^* $ arguments, it yields
\begin{eqnarray*}
\|P_{\complement A_\varepsilon} v\|_{L^{4}_{I_j} L^\infty_{x}}  & \lesssim&
\|D^{-1/4}_x P_{\complement A_\varepsilon}  v(a_j)\|_{L^2} + \|
D_x^{3/4} P_{\complement A_\varepsilon} F \|_{L^1_{I_j} L^2_{x}} \\
 & \lesssim & \varepsilon^{1/8}\|P_{\complement A_\varepsilon}  v(a_j)\|_{L^2} +\varepsilon^{-3/8}  \|
 P_{\complement A_\varepsilon} F \|_{L^1_{I_j} L^2_{x}}\, .
\end{eqnarray*}
Therefore, we get
$$
\|P_{\complement A_\varepsilon} v\|_{L^{1}_{I_j} L^\infty_{x}} \lesssim
\varepsilon^{1/2}\|P_{\complement A_\varepsilon}  v(a_j)\|_{L^2} +  \|
 P_{\complement A_\varepsilon} F \|_{L^1_{I_j} L^2_{x}} 
$$
and summing over $j$,
\begin{eqnarray*}
\|P_{\complement A_\varepsilon} v \|_{L^{1}_{T} L^\infty_{xy}} & \lesssim & \varepsilon^{1/2} \sum_j  \|P_{\complement A_\varepsilon}  v\|_{L^\infty_T
L^2_{x}} +    \|
 P_{\complement A_\varepsilon} F \|_{L^1_{T} L^2_{x}}\, .
\\
& \lesssim & (\varepsilon^{1/2}+T) \| P_{\complement A_\varepsilon}  v \|_{L^\infty_T
L^2_{x}} +   \|
 P_{\complement A_\varepsilon} F \|_{L^1_{T} L^2_{x}}\, .
\end{eqnarray*}
\end{proof}
We now need the following  energy estimate
\begin{lemma}\label{lemma32} 
Let $ s\ge 1 $. There exists $ C>0 $ such that all $ 0<\varepsilon<<1 $ and all $ \varphi\in H^s(\R) $, the   solution $ u\in C(0,T;H^s)  $ of  $(K_\varepsilon) $ with initial data $ \varphi $  satisfies
\begin{equation} \label{energy}
\|P_{\complement A_\varepsilon} u \|_{L^\infty_T H^s}^2 \le \|P_{\complement A_\varepsilon} \varphi\|_{H^s}^2 +C\,\|P_{B_\varepsilon} u_x \|_{L^1_T L^\infty_x}
\|u\|_{L^\infty_T H^s}^2\; .
\end{equation}
\end{lemma}
\begin{proof}
Applying the operator $ P_{\complement A_\varepsilon}  $ on $(K_\varepsilon) $
 and taking the $ H^s $-scalar product  with  $  P_{\complement A_\varepsilon}  u $ we get
$$
\frac{d}{dt} \|P_{\complement A_\varepsilon} u(t)\|_{H^s}^2  =   \int_{\R} J^s_x P_{\complement A_\varepsilon} \partial_x (u^2) J^s_x P_{\complement A_\varepsilon} u\; .
$$
Decomposing  $ u $ as $ u= P_{B_\varepsilon} u + P_{\complement B_\varepsilon} u $ we can rewrite the right-hand side member of the above equality  as  
$$
\int_{\R} J^s_x P_{\complement A_\varepsilon} \partial_x (P_{B_\varepsilon} u)^2  J^s_x P_{\complement A_\varepsilon} u
+ \int_{\R} J^s_x P_{\complement A_\varepsilon}\partial_x \Bigl( ( P_{\complement B_\varepsilon} u)^2
+ 2  P_{B_\varepsilon} u  P_{\complement B_\varepsilon} u\Bigr) J^s_x P_{\complement A_\varepsilon} u : =I_1+I_2 \; .
$$
In the sequel  we will need the following variant of the Kato-Ponce  commutator estimate ( \cite{KP}):
\begin{equation}\label{com}
\Bigl\| \Bigl[ J^s_x P_{\complement A_\varepsilon} ,f] g\Bigr\|_{L^2_x}\lesssim 
\|f_x\|_{L^\infty_x} \|g\|_{H^{s-1}_x}+ \|f\|_{H^s_x} \|g\|_{L^\infty_x} \; .
\end{equation}
Integrating by parts and applying the above commutator estimate we easily estimate the first term by
\begin{eqnarray*}
I_1 & = & 2 \int_{\R} P_{B_\varepsilon} u  \, \partial_x\Bigr(J^s_x P_{\complement A_\varepsilon} 
 P_{B_\varepsilon} u \Bigr) J^s_x P_{\complement A_\varepsilon} u
  + 2 \int_{\R} \Bigl[ J^s_x P_{\complement A_\varepsilon} , P_{B_\varepsilon} u] P_{B_\varepsilon} u_x \, J^s_x P_{\complement A_\varepsilon} u \\
  & \lesssim & \|P_{B_\varepsilon} u_x \|_{L^\infty_x} \|u \|_{H^s}^2 \; ,
\end{eqnarray*}
where, in the last step, we use that according to  the support localization of $ \eta $, 
\begin{equation} \label{supo}
P_{B_\varepsilon} P_{\complement A_\varepsilon}= P_{\complement A_\varepsilon}\; .
\end{equation}
For the second term, we notice that by the frequency projections, all the functions in the integral are supported in frequencies of order $ 1/\sqrt{\varepsilon}$.
Therefore, using Bernstein inequalities  we get 
\begin{eqnarray*}
I_2 & \lesssim &  \varepsilon^{-s-1/2}Ê \Bigl\|  P_{\complement A_\varepsilon} \Bigl( ( P_{\complement B_\varepsilon} u)^2
+ 2  {\mathcal F}^{-1}_x \Bigl(\chi_{\{|\xi|\sim \varepsilon^{-\frac{1}{2}}\}}  {\mathcal F}(P_{B_\varepsilon} u)\Bigr)  P_{\complement B_\varepsilon} u\Bigr)\Bigr\|_{L^1_x}    \| P_{\complement A_\varepsilon}  u \|_{L^\infty_x} \\
& \lesssim & \| P_{\complement A_\varepsilon}  u_x \|_{L^\infty_x} \| u\|_{H^s}^2\; .
\end{eqnarray*}
 \eqref{energy}  then follows  by integration in time, using again \eqref{supo}.
\end{proof}
\noindent
{\bf Proof of Proposition \ref{propo2}}
Applying \eqref{estim1}Ê to $ u_x $ with $ u $ solving $(K_\varepsilon) $ we get
\begin{eqnarray}
\| P_{\complement A_\varepsilon} u_x\|_{L^{1}_T L^\infty_{x}}  & \lesssim & (\varepsilon^{1/2}+T)
\|P_{\complement A_\varepsilon} \,u_x\|_{L^\infty_T\, L^2_{x}}
+\|P_{\complement A_\varepsilon} \partial_x(u^2)\|_{L^1_{T}L^2_{x}} \nonumber\\
 & \lesssim & (\varepsilon^{1/2}+T)
\|u\|_{L^\infty_T\, H^1_{x}} + T \| u\|_{L^\infty_T H^1_x}^2 \label{carlos2} \; .
\end{eqnarray}
 Therefore, gathering \eqref{energy}, \eqref{carlos2} and \eqref{linear2} we obtain
\begin{eqnarray*}
\|P_{\complement A_\varepsilon} u \|_{L^\infty_T H^s_x}^2 & \lesssim &   \|P_{\complement A_\varepsilon} u_0\|_{H^s}^2 +C\,
\|u\|_{L^\infty_T H^s_x}^2 \Bigl(  T^{1/4} \|P_{B_\varepsilon} P_{A_\varepsilon} u_x \|_{L^4_T L^\infty_x} +\| P_{\complement A_\varepsilon} u_x\|_{L^{1}_T L^\infty_{x}} \Bigr) \\
 &\lesssim & \|P_{\complement A_\varepsilon} u_0\|_{H^s}^2 +C\,
(\varepsilon^{1/2}+T^{1/4}) \|u\|_{L^\infty_T H^s_x}^2 \Bigl(  \|u\|_{Y_{\varepsilon,T}^1} +  \|u\|_{Y_{\varepsilon,T}^1}^2)\; ,
\end{eqnarray*}
which completes the proof of \eqref{est1propo2}. 
\qed \vspace{2mm} \\
\section{Proof of Theorem \ref{theo2}}
\subsection{Uniform bound on the solutions}
 Let $ u\in C^\infty(\R;H^\infty(\R)) $ be a solution of $(K_\varepsilon) $. Combining Propositions \ref{pro1} and \ref{propo2} we infer that for any $ s\ge 1 $ and $ T\in ]0,1[$, 
$$
\|u\|_{Y_{\varepsilon,T}^s}^2\le C\,  \| \varphi\|_{H^s}^2 +C\,  (\sqrt{\varepsilon}+T^{1/4})\| u\|_{Y_{\varepsilon,T}^s}^2 \| u\|_{Y_{\varepsilon,T}^1}
\Bigl(1+\| u\|_{Y_{\varepsilon,T}^1}^3 \Bigr)\; ,
$$
for some constant $ C>0 $.
Since $ u $ is smooth, $ T\mapsto \|u\|_{Y_{\varepsilon,T}^s} $ is continuous and 
$\displaystyle \limsup_{T\searrow 0} \|u\|_{Y_{\varepsilon,T}^s} \lesssim\| \varphi \|_{H^s} $.  Therefore a classical continuity argument ensures that for  any $ \delta>0 $ there exists $ \alpha>0 $ such that
\begin{equation} \label{ji}
C   (\sqrt{\varepsilon}+T^{1/4})\| u\|_{Y_{\varepsilon,T}^s}^2 \| u\|_{Y_{\varepsilon,T}^1}
\Bigl(1+\| u\|_{Y_{\varepsilon,T}^1}^3 \Bigr)\le \delta
\end{equation}
and  $ \|u\|_{Y_{\varepsilon,T}^s}\lesssim \| \varphi\|_{H^s} $ provided
\begin{equation}\label{condinitial}
  (\sqrt{\varepsilon}+T^{1/4}) \le \alpha (\| \varphi\|_{H^1}+\| \varphi\|_{H^1}^4)^{-1}\; .
  \end{equation}
  By continuity with respect to initial data (for any fixed $ \varepsilon >0 $) it follows that for any fixed initial data $ \varphi\in H^s(\R) $, $ s\ge 1 $,  the emanating solution $ u\in C(\R;H^s(\R)) $ of $(K_\varepsilon) $, with 
  \begin{equation}\label{eps}
  0< \varepsilon \le \varepsilon_0(\|\varphi\|_{H^1}):=\frac{\alpha^2}{4} (\| \varphi\|_{H^1}+\| \varphi\|_{H^1}^4)^{-2}\; ,
  \end{equation}
  satisfies
  \begin{equation}\label{gf}
  \|u\|_{Y_{\varepsilon,T}^s}\lesssim \| \varphi\|_{H^s} \; ,
  \end{equation}
  with $ T=T(\|\varphi\|_{H^1})\sim (\| \varphi\|_{H^1}+\| \varphi\|_{H^1}^4)^{-4} $.\\
Finally, the result for $ \varepsilon\in [ \varepsilon_0(\|\varphi\|_{H^1}), 1] $ follows from a dilation argument. Indeed, it is easy to check that $ u $ is a solution of $(K_\varepsilon) $ with initial data $ \varphi $
 if and only if  $ u_\lambda=u_\lambda(t,x)=\lambda^{-2} u(\lambda^{-3}t, \lambda^{-1} x) $ is a solution of $ (K_{\lambda^{2} \varepsilon}) $ with initial data
  $ \varphi_{\lambda}=\lambda^{-2} \varphi(\lambda^{-1}x) $. Hence, taking $ \lambda=\varepsilon^{-1/2}\ge 1 $ we observe that $ u_\lambda $ satisfies 
  $ (K_{1}) $. By classical well-posedness result for $ (K_1) $ (see for instance \cite{Kato}), there exists a non increasing function 
   $ R \, :\, \R_+^*\to \R_+^* $ such that 
   $$
   \|u_\lambda\|_{L^\infty_{T'} H^s} \lesssim \| \varphi_\lambda \|_{H^s} \mbox{ with } T'=R(\|\varphi_\lambda\|_{H^1}) \; .
  $$
 Coming back to $ u $, noticing that $ \|\varphi_{\lambda}\|_{H^1} \lesssim \lambda^{-3/2} \|\varphi\|_{H^{1}} $
  and  that $ 1\le \lambda=\varepsilon^{-1/2} \lesssim  (\| \varphi\|_{H^1}+\| \varphi\|_{H^1}^2) $  we deduce   that 
  $$
  \|u\|_{L^\infty_T H^s} \lesssim \| \varphi\|_{H^s} \mbox{ with } T=T(\|\varphi\|_{H^1}) \; ,
  $$
  which completes the proof of \eqref{to3}.
  \subsection{Proof the equi-continuity result}
  Now to prove the equi-continuity result we will make use of  Bona-Smith argument \cite{BS}. To simplify the expository we will only consider 
   the most difficult case that is the case $ s=1 $. We thus want to prove that, be given a sequence $ \{\varphi_k\} \subset H^1(\R) $ converging 
   towards $ \varphi $ in $ H^1(\R) $, the emanating solutions $ u_{\varepsilon,k}:=S_{K_\varepsilon}(\varphi_k) $ satisfy 
   \begin{equation}
   \lim_{k\to \infty} \sup_{0<\varepsilon<1} \| u_{\varepsilon,k} -u_{\varepsilon} \|_{L^\infty_T H^1} =0 \;,
   \end{equation}
   where $ u_\varepsilon:=S_{K_\varepsilon}(\varphi) $ and $ T=T(\|\varphi\|_{H^1}) $. 
  We first notice that we can restrict ourself to consider  $ \varepsilon $ satisfying \eqref{eps} since the same dilation argument as above  yields directly the result  otherwise.
  
    The first step consists in repeating the arguments of Sections \ref{section2} $ \& $ \ref{section3} to get a $ L^2 $-Lipschitz bound, uniform in $ \varepsilon$,  for $ H^1 $-solution. 
    This is the aim of the following proposition which proof is postponed in the appendix. 
      \begin{proposition}\label{proplip}
    Let $0< \varepsilon<1 $, $T>0  $ and $ v \in Y_{\varepsilon,T}^1 $  satisfying 
    \begin{equation}\label{hypv1}
    (\sqrt{\varepsilon}+T^{1/4})\Bigl(\| v\|_{Y_{\varepsilon,T}^1}+\| v\|_{Y_{\varepsilon,T}^1}^4 \Bigr) <\!\!< 1 \; 
    \end{equation}
    and 
    \begin{equation}\label{hypv2}
     \|P_{B_\varepsilon} \partial_x v \|_{L^{1}_{T} L^\infty_{x}} \lesssim  (\sqrt{\varepsilon}+T^{1/4})\Bigl(\| v\|_{Y_{\varepsilon,T}^1}+\| v\|_{Y_{\varepsilon,T}^1}^2 \Bigr) \;  .
     \end{equation}
     Then any solution $ w \in C([0,T]; H^1(\R)) $ 
     to 
   \begin{equation}\label{ww}
     \partial_t w + \partial_x^3 w +\varepsilon \partial_x^5 w +  \frac{1}{2} \partial_x( w v ) =0 
    \end{equation}
  satisfies
      \begin{equation}
    \|w\|_{L^\infty_T L^2_x} \lesssim \|Êw(0)\|_{L^2_x}
    \end{equation}
    where the implicit constant is independent of $ \varepsilon $.
    \end{proposition}
    
   Now, for any $ \varphi\in H^1(\R) $ and any dyadic integer $ N $ we set $ \varphi^N:=P_{\le N} \varphi $.
   By straightforward calculations in Fourier space, for any $\varphi\in H^1(\R) $, any $ N\ge 1 $ and any $ r\ge 0 $, 
   \begin{equation}\label{init}
   \|\varphi^N\|_{H^{1+r}_x}\lesssim N^r \|Ê\varphi\|_{H^1_x}\quad\mbox{Êand } \quad
   \| \varphi^N-\varphi\|_{H_x^{1-r}} \lesssim o(N^{-r}) \| \varphi\|_{H^1_x}\; .
   \end{equation}
    Setting $ u_{\varepsilon}^N:= S_{K_\varepsilon}(\varphi^N ) $ and $ u_{\varepsilon,k}^N:= S_{K_\varepsilon}(\varphi_k^N ) $,   \eqref{ji} ensures that there exists $ T_0=T_0(\|\varphi\|_{H^1}) \in]0,1[ $ such that
     for $ k  $ large enough and $ z:= u_\varepsilon$,  $ u_\varepsilon^N$, $  u_{\varepsilon,k}$ or $  u_{\varepsilon,k}^N$, 
       \begin{equation}\label{tktk}
      (\sqrt{\varepsilon}+T_0^{1/4})\Bigl(\| z\|_{Y_{\varepsilon,T_0}^1}+\| z\|_{Y_{\varepsilon,T_0}^1}^4 \Bigr) <\!\! <1  \;  
     \end{equation}
    and, according to  \eqref{gf}, \eqref{linear2}  and \eqref{carlos2},
    \begin{equation}\label{qa1}
       \|z \|_{Y^1_{\varepsilon,T_0}} \le 2  \| \varphi\|_{H^1} \; \mbox{ and } \quad \|P_{B_\varepsilon}  \partial_x  z \|_{L^{1}_{T_0} L^\infty_{x}} 
    \lesssim  (\sqrt{\varepsilon}+T_0^{1/4})\Bigl(\| z\|_{Y_{\varepsilon,T_0}^1}+\| z\|_{Y_{\varepsilon,T_0}^1}^2 \Bigr)  \; .
    \end{equation}
    Moreover,
   \begin{equation}\label{qas}
    \|u_{\varepsilon}^N \|_{Y^s_{\varepsilon,T_0}} + \| u_{\varepsilon,k}^N \| _{ÊY^s_{\varepsilon,T_0}} \lesssim \| \varphi^N\|_{H^s} \lesssim N^{s-1}  \| \varphi\|_{H^1}
    \end{equation}
    provided $ s\ge 1 $.
    By the triangle inequality, it holds 
    \begin{equation}\label{trtr}
    \|u_\varepsilon-u_{\varepsilon,k}\|_{L^\infty_T H^1_x}Ê\le 
     \|u_\varepsilon-u_{\varepsilon}^N\|_{L^\infty_T H^1_x}+ \|u_{\varepsilon}^N-u_{\varepsilon,k}^N\|_{L^\infty_T H^1_x}+ 
     \|u_{\varepsilon,k}^N-u_{\varepsilon,k}\|_{L^\infty_T H^1_x}\; .
     \end{equation}
    We start by  estimating the first term of the right-hand side fo \eqref{trtr}. Setting $ w_\varepsilon:=u_\varepsilon-u_{\varepsilon}^N$, we observe that $ w_\varepsilon$ satisfies
     \begin{equation}\label{eq1w}
    \partial_t w_\varepsilon + \partial_x^3 w_\varepsilon +\varepsilon \partial_x^5 w_\varepsilon +  \frac{1}{2} \partial_x( w_\varepsilon (u_\varepsilon^N+u_\varepsilon)) =0 \; .
    \end{equation}
    Therefore, combining Proposition \ref{proplip}, \eqref{tktk}-\eqref{qa1} and \eqref{init} we get that 
    \begin{equation}\label{ouaoua}
    \|w_\varepsilon\|_{L^\infty_{T_0} L^2_x} \lesssim o(N^{-1})  \; .
    \end{equation}    
    According to \eqref{est2pro1} we also have 
      \begin{equation}\label{zw2}
    \|P_{A_\varepsilon}w_\varepsilon\|_{X^{1,1/2,1}_{\varepsilon,T_0}} \le C\,   \|\varphi-\varphi^N\|_{H^1_x} + \frac{1}{2}\|w_\varepsilon\|_{Y^{1}_{\varepsilon,T_0}} \; .
    \end{equation}   
    Now to estimate $ P_{\complement A_\varepsilon}  $ we rewrite the equation satisfying by $ w_\varepsilon$ in the  following less symmetric  way :
    $$
      \partial_t w_\varepsilon + \partial_x^3 w_\varepsilon +\varepsilon \partial_x^5 w_\varepsilon =-  \frac{1}{2} \partial_x( w_\varepsilon^2) - \partial_x( u_\varepsilon^N w_\varepsilon ) \; .
    $$
    Applying the operator $ P_{\complement A_\varepsilon}  $ on the above equation 
 and taking the $H^1 $-scalar product  with  $  P_{\complement A_\varepsilon}  w_\varepsilon $ we get
 \begin{eqnarray}
\frac{d}{dt} \|P_{\complement A_\varepsilon} w_\varepsilon(t)\|_{H^1_x}^2  & = &   \int_{\R}  J_x^1P_{\complement A_\varepsilon} \partial_x (w_\varepsilon^2)  J_x^1 P_{\complement A_\varepsilon} w_\varepsilon+   2 \int_{\R}  J_x^1 P_{\complement A_\varepsilon} (u_\varepsilon^N \partial_x w_\varepsilon) J_x^1 P_{\complement A_\varepsilon} w_\varepsilon \nonumber \\
 & & 
+ 2\int_{\R} J_x^1  P_{\complement A_\varepsilon} (w_\varepsilon \partial_x u_\varepsilon^N ) J_x^1  P_{\complement A_\varepsilon} w_\varepsilon\;  .
\label{cc}
\end{eqnarray}
    The contribution of the first term of the above right-hand side can be estimated in exactly the same way as in the proof of Lemma \ref{lemma32} by 
    $
    \| P_{B_\varepsilon} \partial_x w_\varepsilon \|_{L^\infty_x} \|Êw_\varepsilon\|_{H^1_x}^2  $.
    The second term can be estimated also in the same way by 
    $$
   ( \| P_{B_\varepsilon} \partial_x u_\varepsilon^N \|_{L^\infty_x} \|Êw_\varepsilon\|_{H^1_x} + \|u_\varepsilon^N\|_{H^1_x} \|P_{B_\varepsilon} \partial_x w_\varepsilon \|_{L^\infty_x})
    \|Êw_\varepsilon\|_{H^1_x} \; .
    $$
     The difficulty comes from the third term. To estimate its contribution we first decompose $ w_\varepsilon$ and $ u_\varepsilon^N $ to rewrite it as 
    $$
2 \int_{\R} J^1_x P_{\complement A_\varepsilon} \Bigl(  P_{\complement B_\varepsilon} w_\varepsilon  P_{\complement B_\varepsilon} \partial_x u_\varepsilon^N 
+P_{\complement B_\varepsilon} w_\varepsilon   P_{B_\varepsilon}\partial_x  u_\varepsilon^N  
+  P_{B_\varepsilon} w_\varepsilon  P_{\complement B_\varepsilon} \partial_x u_\varepsilon^N \Bigr) J^1_x P_{\complement A_\varepsilon} w_\varepsilon
$$
  $$
+2\int_{\R} J^1_x P_{\complement A_\varepsilon}  (P_{B_\varepsilon} w_\varepsilon P_{B_\varepsilon} \partial_x u_\varepsilon^N ) J^1_x P_{\complement A_\varepsilon} w_\varepsilon
=I_1+I_2
$$
According to the frequency projections, in the same way as proof of Lemma \ref{lemma32}, all the functions in $ I_1 $ are supported in frequencies of order 
 $ \varepsilon^{-1/2} $, which leads to 
 $$
 I_1 \lesssim \|P_{\complement A_\varepsilon}Ê\partial_x w_\varepsilon \|_{L^\infty_x} \|w_\varepsilon\|_{H^1_x} \|u_\varepsilon^N\|_{H^1_x} \; .
 $$
 Finally we control the contribution of $ I_2 $ by 
 $$
 I_2\lesssim \|P_{\complement A_\varepsilon}\partial_x Êw_\varepsilon \|_{L^\infty_x}\Bigl( \|Êw_\varepsilon\|_{H^1} \|u_\varepsilon^N\|_{H^1_x}+
 \| w_\varepsilon\|_{L^2} \| u_\varepsilon^N\|_{H^2} \Bigr)
 $$
 Note that the difficulty  to control $ I_2$ comes from the fact  that we can not avoid to put a $ H^2$-norm on $u_\varepsilon^N$. But the idea of Bona-Smith is to compensate the growth with $ N $ of this $ H^2 $-norm by the decay with $ N $ of the $ L^2 $-norm of $w_\varepsilon$.
 Actually, integrating \eqref{cc} Êin time, with the above estimates together with \eqref{ouaoua} and \eqref{qa1}-\eqref{qas} in hand, we  get
 \begin{eqnarray}
 \|P_{\complement A_\varepsilon} w_\varepsilon \|_{L^\infty_{T_0} H^1_x}^2 & \lesssim & \|P_{\complement A_\varepsilon} (\varphi-\varphi^N)\|_{H^1_x}^2 
 +(\|P_{B_\varepsilon} \partial_x w_\varepsilon \|_{L^1_{T_0} L^\infty_x}+\|P_{B_\varepsilon} \partial_x  u_\varepsilon^N\|_{L^1_{T_0} L^\infty_x})
\|w_\varepsilon\|_{L^\infty_{T_0} H^1_x}^2 \nonumber \\
 & & +\|P_{B_\varepsilon}  \partial_x  w_\varepsilon \|_{L^1_{T_0} L^\infty_x} (\|w_\varepsilon\|_{L^\infty_{T_0} H^1_x}\|u_\varepsilon^N\|_{L^\infty_{T_0} H^1_x}+
 \|w_\varepsilon\|_{L^\infty_{T_0} L^2_x}\|u_\varepsilon^N\|_{L^\infty_{T_0} H^2_x})\nonumber \\
 & \le& \gamma_1(N)+ \frac{1}{8} \|w_\varepsilon\|_{L^\infty_{T_0} H^1_x}^2+ 2\|\varphi\|_{H^1_x} \|P_{B_\varepsilon}
  \partial_x  w_\varepsilon\|_{L^1_{T_0} L^\infty_x}   \|w_\varepsilon\|_{L^\infty_{T_0} H^1_x} \nonumber \\
   & & +o(N^{-1} )  N 
 \label{zw}
\end{eqnarray}
where $ \gamma_1(N)\to 0 $ as $ N\to \infty $.
On the other hand, applying Lemma \ref{carlos} on the  $x$-derivative of  \eqref{eq1w}  we get 
\begin{equation} \label{carlos3} 
\| P_{\complement A_\varepsilon}  \partial_x  w_\varepsilon\|_{L^{1}_{T_0} L^\infty_{x}} 
   \lesssim   (\varepsilon^{1/2}+T_0)
\|w_\varepsilon\|_{L^\infty_{T_0}\, H^1_{x}} + T_0 \| w_\varepsilon\|_{L^\infty_{T_0} H^1_x}(\|u_\varepsilon\|_{L^\infty_{T_0} H^1}+ \|u_\varepsilon^N\|_{L^\infty_{T_0} H^1})
\end{equation}
 Therefore, gathering \eqref{zw}, \eqref{carlos3} and \eqref{zw2} with \eqref{tktk} in hand, we obtain
$$
\| w_\varepsilon \|_{Y^1_{T_0}}^2  \le   \gamma_2 (N)+\frac{1}{2} \| w_\varepsilon \|_{Y^1_{T_0}}^2
$$
with $ \gamma_2 (N)\to 0 $ as $ N\to \infty $. This ensures that  
$$
 \|u_\varepsilon-u_{\varepsilon}^N\|_{L^\infty_{T_0}\, H^1_{x}}  \le  2   \gamma_2 (N) \; .
 $$
To estimate the contribution of the third term of the right-hand side of \eqref{trtr} we proceed exactly in the same way as for the first one, by replacing 
 $ u_\varepsilon$ by $u_{\varepsilon,k} $ and $u_{\varepsilon}^N$ by $ u_{\varepsilon,k}^N $.  We then obtain 
  $$
 \|u_{\varepsilon,k}-u_{\varepsilon,k}^N\|_{L^\infty_{T_0}\, H^1_{x}}  \le \gamma_3 (N) \; .
 $$
 with  $ \gamma_3 (N)\to 0 $ as $ N\to \infty $. 
 Finally, the contribution of the second term of the right-hand side of \eqref{trtr} is also obtain  in the same way by replacing 
 $ u_\varepsilon$ by $u_{\varepsilon,k}^N $  (actually, contrary to the preceding contributions,  here both terms $u_{\varepsilon,k}^N$ and $ u_{\varepsilon}^N $   can play a symmetric role ). However,  for this term,   Proposition \ref{proplip} only ensures that 
 $$
 \| u_{\varepsilon}^N-u_{\varepsilon,k}^N\|_{L^\infty_{T_0} L^2_x}\lesssim  \|Ê\varphi-\varphi_k\|_{L^2_x} \; .
 $$
 Therefore, setting $ w_\varepsilon= u_{\varepsilon}^N-u_{\varepsilon,k}^N$,  one has to replace $  o(N^{-1} )  N $ by $  \|Ê\varphi-\varphi_k\|_{L^2_x}  N $ in the right-hand side member of \eqref{zw} when estimating 
  $  \| P_{\complement A_\varepsilon}  (u_{\varepsilon}^N-u_{\varepsilon,k}^N)\|_{L^\infty_{T_0} H^1_x} $.
   We thus obtain
 $$
 \|u_{\varepsilon}^N-u_{\varepsilon,k}^N\|_{L^\infty_{T_0}\, H^1_{x}}  \lesssim \|Ê\varphi-\varphi_k\|_{H^1_x}
 +   N  \|Ê\varphi-\varphi_k\|_{L^2_x}  \; .
 $$
 Gathering the above estimates, \eqref{trtr} leads to  
$$
\lim_{k\to +\infty} \sup_{0<\varepsilon<\varepsilon_0(\|\varphi\|_{H^1})} \|u_{\varepsilon}-u_{\varepsilon,k}\|_{L^\infty_{T_0}\, H^1_{x}}  =0
$$
which completes the proof of Theorem \ref{theo2}.
  \subsection{Proof of Theorem \ref{theo1}}
  We follow general arguments (see for instance  \cite{GW}).   Let us denote by $ S_{K_\varepsilon} $ and $ S_{KdV}Ê$ the nonlinear
  group associated with respectively $(K_\varepsilon) $ and KdV. Let  $ \varphi\in H^s_x(\R) $, $s\ge 1$ and let $ T=T(\|\varphi\|_{H^1_x} ) >0 $ be given by Theorem \ref{theo1}. For any $ N>0 $  we can rewrite $ S_{K_\varepsilon}(\varphi)-
  S_{KdV}(\varphi) $ as
  \begin{eqnarray*}
  S_{K_\varepsilon}(\varphi)-S_{KdV}(\varphi)&=& \Bigl(S_{K_\varepsilon}(\varphi)-S_{K_\varepsilon}(P_{\le N} \varphi)\Bigr)+
  \Bigl(S_{K_\varepsilon}(P_{\le N} \varphi)- S_{KdV}(P_{\le N}\varphi)\Bigr)\\
  & & + \Bigl(S_{KdV}(P_{\le N}\varphi)-S_{KdV}(\varphi)\Bigr)=I_{\varepsilon,N}+J_{\varepsilon,N}+K_{N}\; .
  \end{eqnarray*}
  By continuity with respect to initial data in $ H^s(\R) $ of the solution map associated with the KdV equation, we have  $\displaystyle  \lim_{N\to \infty}Ê
  \|K_N\|_{L^\infty(0,T;H^s_x) }=0 $. On the other hand, \eqref{to4} ensures that
  $$
  \lim_{N\to \infty} \sup_{\varepsilon\in]0,1[} \|I_{\varepsilon,N}\|_{L^\infty(0,T;H^s_x)}Ê=0 \; .
  $$
  It thus remains to check that for any fixed $ N>0 $,  $ \displaystyle \lim_{\varepsilon\to 0} \|J_{\varepsilon,N}\|_{L^\infty(0,T;H^s_x)}Ê=0 $.
  Since $ P_{\le N} \varphi\in H^\infty(\R) $, it is worth noticing that
   $ S_{K_\varepsilon}(P_{\le N} \varphi) $ and $ S_{KdV}(P_{\le N}\varphi) $ belong to $C^\infty (\R;H^\infty(\R)) $. Moreover, according to Theorem \ref{theo2} and the well-posedness theory of the KdV equation (see for instance \cite{Bourgain1993}),  for all $ \theta\in \R$ and $ \varepsilon\in ]0,1[ $,
   $$
   \|S_{K_\varepsilon}(P_{\le N} \varphi) \|_{L^\infty_T H^{\theta}_x}+ \|S_{KdV}(P_{\le N} \varphi)\|_{L^\infty_T H^{\theta}_x} \le C(N,\theta,  \|\varphi\|_{L^2_x})  \; .
   $$

 Now, setting   $ v_\varepsilon := S_{K_\varepsilon}(P_{\le N} \varphi)  $ and  $v:=S_{KdV}(P_{\le N}\varphi) $, we  observe that   $ w_\varepsilon:= v_\varepsilon-v $ satisfies
   $$
    \partial_t w_\varepsilon+ \partial_x^3 w_\varepsilon + \varepsilon \partial_x^5   w_\varepsilon = -\frac{1}{2}\partial_x \Bigl( w_\varepsilon
    (v +v_\varepsilon)\Bigr) - \varepsilon v_{5x}
   $$
   with initial data $ w_\varepsilon(0)=0 $. Taking the  $ H^s $-scalar product of this last equation with $w_\varepsilon $ and  integrating by parts 
   we get 
   $$
   \frac{d}{dt} \|w_\varepsilon\|_{H^s_x} \lesssim \Bigl( 1+\|\partial_x(v +v_\varepsilon)\|_{L^\infty_x}) \| w_\varepsilon\|_{H^s_x}^2
   + \|[J^s, (v +v_\varepsilon)]\partial_x  w_\varepsilon\|_{L^2} \| w_\varepsilon\|_{H^s_x} + \varepsilon^2 \|v_{5x}\|_{H^s_x}^2 \; .
   $$
 Making use of   the following  commutator estimate (see for instance \cite{Tom}), that  holds for $ s>1/2$, 
   \begin{equation}\label{commute}
   \| [J^s_x, f] g \|_{L^2_x}Ê\lesssim \|Êf_x\|_{H^{s}_x} \| g\|_{H^{s-1}_x}\; ,
   \end{equation}
   we easily get 
   $$
   \frac{d}{dt} \| w_\varepsilon(t)\|_{H^s_x}^2\lesssim  C(N,s+1,  \|\varphi\|_{L^2_x}) 
   \| w_\varepsilon(t)\|_{H^s_x}^2+ \varepsilon^2  C(N,5+s,  \|\varphi\|_{L^2_x})^2\; .
   $$
   Integrating this differential inequality on $ [0,T]$, this ensures  that $  \displaystyle \lim_{\varepsilon\to 0} \| w_\varepsilon\|_{L^\infty(0,T;H^s) }=0 $ and completes the proof of Theorem  \ref{theo1} with $T=T(\|\varphi\|_{H^1})$. Finally, recalling that the energy conservation of the KdV equation ensures that for any $ \varphi\in H^1(\R) $ it holds,
  $$
\sup_{t\in\R}  \|S_{KdV} (\varphi)(t)\|_{H^1_x}\lesssim \|\varphi\|_{H^1_x} + \|\varphi\|_{L^2_x}^5 \; ,
$$
  we obtain the same convergence result on any time interval $ [0,T_0] $ with $ T_0>T(\|\varphi\|_{H^1_x})$ by reiterating the convergence result   about  $ T_0¡/T(\|\varphi\|_{H^1_x}+ \|\varphi\|_{L^2_x}^5) $ times.
  \section{appendix: Proof of Proposition \ref{proplip}}
   We follow very closely Sections \ref{section2} and \ref{section3}. The first step consists in  establishing the following  estimate on $ P_{A_\varepsilon} w $.
    \begin{proposition}\label{proA1}
   Let $ 0<T<1 $ and $ w \in C([0,T]; H^1(\R)) $ be a solution to \eqref{ww} with $ 0<\varepsilon<\! <1 $ and initial data $ \varphi $. Then it holds
  \begin{equation}\label{est1pro1A}
  \| P_{A_\varepsilon} w \|_{X^{0,1/2,1}_{\varepsilon,T} }\lesssim  \|\varphi\|_{L^2}  +T^{\frac{1}{4}-}\| v\|_{Y_{\varepsilon,T}^1} \| w\|_{Y_{\varepsilon,T}^0}
  (1+ \| v\|_{Y_{\varepsilon,T}^1})
  \end{equation}
\end{proposition}
 \begin{proof}  We proceed as in Section \ref{section2}. First we observe that we have trivially 
   \begin{equation}\label{jj}
  \|P_{\le 8 } \partial_x (vw) \|_{X^{0,-1/2,1}_{\varepsilon}} \lesssim
   \| v w \|_{L^2_{tx}} \lesssim T^{1/2} \| v\|_{L^\infty_t H^1_x} \| w\|_{L^\infty_t L^2_x}  \; .
  \end{equation}
  and 
  \begin{eqnarray}
  \|P_{A_\varepsilon } \partial_x (v P_{\le 8} w) \|_{X^{0,-1/2,1}_{\varepsilon}}   & \lesssim  &
  \|P_{A_\varepsilon } \partial_x (v P_{\le 8} w) \|_{L^2_{tx}} \nonumber \\
  & \lesssim & T^{1/2} \Bigl(  \|v_x\|_{L^\infty_t L^2_x} \|P_{\le 8} w \|_{L^\infty_{tx}}+
   \|v\|_{L^\infty_{tx}} \|w\|_{L^\infty_t L^2_x} \Bigr)  \nonumber \\
   & \lesssim & T^{1/2} \|v\|_{L^\infty_t H^1_x} \|w\|_{L^\infty_t L^2_x}\label{jjA}\; .
  \end{eqnarray}
  Now to control $ \|P_{A_\varepsilon } \partial_x (w P_{\le 8} v) \|_{X^{0,-1/2,1}_{\varepsilon}}  $ we notice that in the same way as in 
   \eqref{jjA} we have 
   $$
 \|P_{A_\varepsilon } \partial_x (P_{\le 16} w P_{\le 8} v) \|_{X^{0,-1/2,1}_{\varepsilon}}\lesssim   T^{1/2} \| v\|_{L^\infty_t H^1_x} \| w\|_{L^\infty_t L^2_x}  \; .
   $$
 On the other hand, according to the frequency projections and Lemma \ref{lemme1}, the contribution of $P_{\ge 16} w $ can be estimated by 
 \begin{eqnarray}\label{eqeq2}
\|\partial_x P_{A_\varepsilon}(P_{\le 8} v  P_{\ge 16} w) \|_{X^{0,-1/2,1}_\varepsilon } & =& 
\|\partial_x P_{A_\varepsilon}\Lambda(P_{\le 8}v ,  P_{\ge 16} w) \|_{X^{0,-1/2,1}_\varepsilon }\nonumber \\
 & \lesssim & 
\Bigl\|\partial_x P_{A_\varepsilon}\Lambda\Bigl(P_{\le 8}{\mathcal F}^{-1}_{xt}(|\widehat{v}|) ,  \partial_x^{-1} P_{\ge 16} {\mathcal F}^{-1}_{xt}(|\widehat{w}|)\Bigr) \Bigr\|_{X^{1,-1/2,1}_\varepsilon }\nonumber \\
& \lesssim  &  \|v\|_{X^{0,1}_\varepsilon}
 \Bigl( T^{1/4}(\|P_{A_\varepsilon} w\|_{X^{0,1/2,1}_\varepsilon}+ \| w\|_{X^{-1,1}_\varepsilon})
 +\| w\|_{L^2_{tx}} \Bigr)
  \; .
 \end{eqnarray}
 To continue we need the following variant of Lemma \ref{lemme2}.
 \begin{lemma}\label{lemmeA2}
 Let $ v$ and $ w $ be two smooth functions supported in time in $ ]-T,T[ $ with $ 0<T\le 1 $.  Then, in the region  where the  strong resonance relation  \eqref{strong} holds, we have
 \begin{eqnarray}\label{eqlemmeA2}
 \|\partial_x ÊP_{A_\varepsilon} P_{\ge 8}(P_{\ge 8} v P_{\ge 8} w ) \|_{X^{0,-1/2,1}_\varepsilon } & \lesssim  & 
 T^{\frac{1}{4}-}  \|v\|_{X^{0,1}_\varepsilon}   \|w\|_{X^{-1,1}_\varepsilon}+\|v_x\|_{L^2_{tx}} ( \|w\|_{X^{-1,1}_\varepsilon}+ \|w\|_{L^2_{tx}}) \nonumber \\
  & &  +\|w\|_{L^2_{tx}}
  ( \|v\|_{X^{0,1}_\varepsilon} + \|v_x\|_{L^2_{tx}})  \; .
 \end{eqnarray}
 \end{lemma}
\begin{proof} We notice that the norms in the right-hand side of \eqref{lemme2}   only see the size of the modulus of the Fourier transforms. We can thus assume that all our functions have non-negative Fourier transforms.
 We set $ I:=  \|\partial_x P_{A_\varepsilon}P_{\ge 8}(P_{\ge 8} v P_{\ge 8} w) \|_{X^{0,-1/2,1}_\varepsilon }  $  and separate different subregions .\\
  $\bullet $ $|\sigma_2|\ge 2^{-5}  |\xi \xi_1 (\xi-\xi_1) |$. Then direct calculations  give
  \begin{eqnarray*}
  I & \lesssim  & 
  T^{\frac{1}{2}-}  \| w \|_{X^{-1,1}_\varepsilon} \| D_x^{-1} P_{\ge 8} v \|_{L^\infty_{tx}}\\
  & \lesssim &  T^{\frac{1}{2}-}  \| v \|_{X^{0,1}_\varepsilon} \| w \|_{X^{-1,1}_\varepsilon} \; .
  \end{eqnarray*}
   $\bullet $ $|\sigma_1|\ge 2^{-5}  |\xi \xi_1 (\xi-\xi_1) |$. 
     Then , by \eqref{linear2}Ê of Lemma \ref{linear} and duality, we get 
  \begin{eqnarray*}
  I & \lesssim  &  \Bigl\| ÊP_{A_\varepsilon} P_{\ge 8}\Bigl(P_{\ge 8} D_x^{-1} {\mathcal F}^{-1}_{tx} (\langle \sigma_1\rangle \widehat{u_1}) P_{\ge 8} D_x^{-1}w \Bigr) \Bigr\|_{L^{\frac{4}{3}+}_t L^{1+}_x} \\
  & \lesssim  & T^{\frac{1}{4}-}  \| v\|_{X^{0,1}_\varepsilon} \| P_{\ge 8} D_x^{-1}w  \|_{L^\infty_t L^{2+}_x}  \\
  & \lesssim &  T^{\frac{1}{4}-}  \| v \|_{X^{0,1}_\varepsilon} \| w \|_{X^{-3/4,3/4}_\varepsilon}  \\
  & \lesssim &  T^{\frac{1}{4}-}  \| v \|_{X^{0,1}_\varepsilon}( \| w \|_{X^{-1,1}_\varepsilon}+\|w\|_{L^2_{tx}})  \; .
  \end{eqnarray*}
        $\bullet $ $|\sigma|\ge 2^{-5}  |\xi \xi_1 (\xi-\xi_1) | $ and $ \max (|\sigma_1|, \sigma_2|) \le  2^{-5}  |\xi \xi_1 (\xi-\xi_1) |$. \\
     Then we separate two subregions. \\
   1. $ |\xi_1|\wedge |\xi_2|\ge 2^{-7} |\xi| $. Then $ |\xi_1|Ê\sim |\xi_2|\gtrsim |\xi| $ and taking $ \delta>0 $ close enough to $ 0 $ we get
    \begin{eqnarray*}
  I & \lesssim  & \|\partial_x P_{A_\varepsilon}P_{\ge 8}(P_{\ge 8} v P_{\ge 8} w) \|_{X^{0,-1/2+\delta}_\varepsilon } \\
  & \lesssim &  \Bigl\| w\, D_x^{-1/2+3\delta}P_{\ge 8} v  \Bigr\|_{L^2} \\
  & \lesssim &  \|D_x^{-1/2+3\delta}P_{\ge 8} v\|_{L^\infty_{tx}} \|w \|_{L^2_{tx}} \\
  & \lesssim &     \|v\|_{X^{1/4,3/4}_\varepsilon}  \|w \|_{L^2_{tx}} \\
    & \lesssim &    ( \|v\|_{X^{0,1}_\varepsilon} + \|\partial_x v \|_{L^2_{tx}} )  \|w\|_{L^2_{tx}} \; .
  \end{eqnarray*}
2. $  |\xi_1|\wedge |\xi_2| < 2^{-7} |\xi| $.  Then, we use that \eqref{vz}Ê holds on the support of $  \eta_{A_\varepsilon} $. In the subregion  $ |\xi_1|\wedge |\xi_2|=|\xi_1| $ we    write
   \begin{eqnarray*}
  I^2 & \lesssim & \sum_{N\ge 4} \Bigl( \sum_{4\le N_1\le 2^{-5}N}
   \Bigl\| \eta_{N}(\xi) \eta_{A_\varepsilon}(\xi)  |\xi|Ê \chi_{\{|\sigma|\sim  \max(N_1 N^2 ,  \varepsilon N^4 N_1 ) \} }{\mathcal F}_x \Bigl(
   (w P_{N_1} v   )\Bigr) \Bigr\|_{X^{0,-1/2,1}_\varepsilon }\Bigr)^2 \\
     & \lesssim  &  \sum_{N\ge 4} \Bigl( \sum_{4\le N_1\le 2^{-5}N} \| P_{N_1} D_x^{-1/2} v \|_{L^\infty_{tx}} \|\chi_{\{|\xi|\sim N\}}  \,\widehat{w}\|_{L^2_{ \tau,\xi}}\Bigr)^2 \\
& \lesssim  &    \sum_{N\ge 4}   \|\chi_{\{|\xi|\sim N\}}  \,\widehat{w}\|_{L^2_{ \tau,\xi}}^2 \Bigl( \sum_{4\le N_1\le 2^{-5}N} N_1^{-1/4} \| P_{N_1} D_x^{1/4}v \|_{L^\infty_{t}L^2_x}   \Bigr)^2 \\
  & \lesssim &     \|v\|_{X^{1/4,3/4}}^2  \|w \|_{L^2_{tx}}^2 \\
    & \lesssim &    ( \|v\|_{X^{0,1}} + \|\partial_x v \|_{L^2_{tx}} )^2  \|w \|_{L^2_{tx}}^2 \; .
  \end{eqnarray*}
  Finally, in the subregion  $ |\xi_1|\wedge |\xi_2|=|\xi_2| $ we    write
   \begin{eqnarray*}
  I^2 & \lesssim & \sum_{N\ge 4} \Bigl( \sum_{4\le N_2\le 2^{-5}N}
   \Bigl\| \eta_{N}(\xi) \eta_{A_\varepsilon}(\xi)  |\xi|Ê \chi_{\{|\sigma|\sim  \max(N_2 N^2 ,  \varepsilon N^4 N_2 ) \} }{\mathcal F}_x \Bigl(
   (v P_{N_2} w   )\Bigr) \Bigr\|_{X^{0,-1/2,1}_\varepsilon }\Bigr)^2 \\
     & \lesssim  &  \sum_{N\ge 4} \Bigl( \sum_{4\le N_2\le 2^{-5}N} \| P_{N_1} D_x^{-3/2}w \|_{L^\infty_{tx}} \|\chi_{\{|\xi|\sim N\}}  \xi \,\widehat{v}\|_{L^2_{ \tau,\xi}}\Bigr)^2 \\
& \lesssim  &    \sum_{N\ge 4}   \|\chi_{\{|\xi|\sim N\}} \xi \,\widehat{v}\|_{L^2_{ \tau,\xi}}^2 \Bigl( \sum_{4\le N_2\le 2^{-5}N} N_2^{-1/4} \| P_{N_2} D_x^{-3/4} w \|_{L^\infty_{t}L^2_x}   \Bigr)^2 \\
  & \lesssim &     \|w\|_{X^{-3/4,3/4}_\varepsilon}^2  \|\partial_x v \|_{L^2_{tx}}^2 \\
    & \lesssim &    ( \|w\|_{X^{-1,1}_\varepsilon} + \| w \|_{L^2_{tx}} )^2  \|\partial_x v \|_{L^2_{tx}}^2 \; .
  \end{eqnarray*}
  \end{proof}
  Now we are in position to prove the main bilinear estimates :
  \begin{lemma}\label{lemA}
 \begin{eqnarray}
 \| P_{A_\varepsilon} \partial_x (v w ) \|_{X^{0,-1/2,1}_\varepsilon}
  &\lesssim & 
  T^{\frac{1}{4}-}\Bigl(  \|w \|_{Y^0_\varepsilon}+  \|w \|_{X^{-1,1}_\varepsilon}\Bigr)
  \Bigl(  \|v \|_{Y^1_\varepsilon}+  \|v \|_{X^{0,1}_\varepsilon}\
  \label{bilinearA}  \Bigr) \; ,
 \end{eqnarray}
 where the functions $ u $ and $ v$ are supported in time in $ ]-T,T[ $ with $ 0<T\le 1 $.   
  \end{lemma}
  \begin{proof}
First, according to \eqref{jj}-\eqref{eqeq2}  and  to the support of $ \eta_{A_\varepsilon} $ it suffices to consider
$$
I:=\Bigr[\sum_{N\ge 4 }  \Bigl( \sum_{L} L^{-1/2}\Big\| \eta_{L}(\sigma) \eta_{N}(\xi) \int_{\R^2} \sum_{N_1\wedge N_2\ge 8} \widehat{P_{N_1} v}Ê(\xi_1,\tau_1)
 \widehat{P_{N_2} w}(\xi_2,\tau_2) \, d\tau_1 \, d\xi_1 \Bigr\|_{L^2_{\tau,\xi} (|\xi|\not \in J_\varepsilon)}\Bigr)^{2}\Bigr]^{1/2}
\; ,
 $$
 where  
 $J_\varepsilon$ is defined in \eqref{defJ}. We consider  different contributions to $ I$.
    \begin{enumerate}
 \item[1.]    $ N_1\wedge N_2< 2^{-10} (N_1\vee N_2) $. Then  it holds 
 $$
 (1-2^{-7})\xi^2 \le \xi^2 -\xi_1(\xi-\xi_1) \le (1+2^{-7})\xi^2 \
 $$
 and it  is  easy to check that $  \Gamma(\xi,\xi_1) \ge 2^{-5} $ as soon as $ |\xi|\not \in  J_\varepsilon $. According to \eqref{resonance} this ensures that \eqref{strong} holds.
 \item[2.]  $ N_1\wedge N_2\ge 2^{-10} (N_1\vee N_2) $. Then $ N_1\sim N_2 \gtrsim N $.
 \begin{enumerate}
\item[2.1.]  The subregion $|\xi|\not\in  \Bigl[ \sqrt{\frac{17}{80 \varepsilon}}, \sqrt{\frac{2}{5 \varepsilon}} \Bigr] $. In this region, by \eqref{linear2}Ê of Lemma \ref{linear} and duality,
  we get
 \begin{eqnarray*}
 I & \lesssim & \sum_{N_1\wedge N_2 \ge 8, \, N_1\sim N_2}\| D_x^{-\frac{1}{4}+} \partial_x( P_{N_1} v P_{N_2} w) \|_{L^{\frac{4}{3}+}_t L^{1+}_x} \\
 & \lesssim &  \sum_{N_1\wedge N_2 \ge 8, \, N_1\sim N_2} T^{\frac{3}{4}-}
  N_1^{-\frac{1}{4}+}\| \partial_x P_{N_1} v \|_{L^\infty_t L^{2+}_x}  \|    P_{N_2} w \|_{L^\infty_t L^{2}_x}\\
 & \lesssim & T^{\frac{3}{4}-} \|v \|_{L^\infty_t H^1}\|w \|_{L^\infty_t L^2_x}\; .
 \end{eqnarray*}
 \item[2.2.] The subregion $|\xi|\in  \Bigl[ \sqrt{\frac{17}{80 \varepsilon}}, \sqrt{\frac{2}{5 \varepsilon}} \Bigr] $.
 \begin{enumerate}
  \item[2.2.1] The subregion  $ |\xi_1|\wedge |\xi_2| \le \sqrt{\frac{17}{80 \varepsilon}} $. Since both cases can be treated in the same way, 
  we assume $  |\xi_1|\wedge |\xi_2|=|\xi_1| $. Then, according to  \eqref{linear2}  and the support of $ \eta_{A_\varepsilon} $ and  Ê$ \eta_{B_\varepsilon} $, we get
  \begin{eqnarray*}
 I & \lesssim &  \sum_{N_1\wedge N_2 \ge 8, \, N_1\sim N_2}
 T^{\frac{1}{2}-}\|  \partial_x( P_{B_\varepsilon}P_{A_\varepsilon}  P_{N_1} v P_{N_2}w) \|_{L^2_{tx}} \\
 & \lesssim &  T^{\frac{1}{2}-} \sum_{N_1\wedge N_2 \ge 8, \, N_1\sim N_2} 
\| P_{B_\varepsilon}P_{A_\varepsilon}   \partial_x P_{N_1} v \|_{L^4_t L^\infty_x}  \|   P_{N_2} w \|_{L^\infty_t L^{2}_x}\\
  & \lesssim &   T^{\frac{1}{2}-}\sum_{N_1\wedge N_2 \ge 8, \, N_1\sim N_2} N_1^{-1/4}\|P_{A_\varepsilon}   P_{N_1} v \|_{X^{1,1/2,1 }_\varepsilon}  \|  P_{N_2} w \|_{L^\infty_t L^{2}_x}\\
 & \lesssim & T^{\frac{1}{2}-} \|P_{A_\varepsilon} v \|_{X^{1,1/2,1 }_\varepsilon}\|w \|_{L^\infty_t L^2_x}\; .
 \end{eqnarray*}
\item[2.2.2] The subregion $ |\xi_1|\wedge |\xi_2| > \sqrt{\frac{17}{80 \varepsilon}} $.  Then as in the proof of  \eqref{bilinear} in Section \ref{section3} we observe that \eqref{strong} holds.
 \end{enumerate}
 \end{enumerate}
 \end{enumerate}
 \end{proof} 
 To complete the proof of Proposition \ref{proA1} we notice that, similarly to Lemma \ref{oo}, one can easily prove that any solution $ w\in C([0,T];L^2(\R)) $ with $ 0<T<1 $ of 
 \eqref{ww} satisfies
   \begin{equation}\label{zeA}
  \|w\|_{X^{-1,1}_{\varepsilon,T}} \lesssim  \|w\|_{L^\infty_T H^{-1}_x} + \|v \|_{L^\infty_T H^1_x}
    \|w \|_{L^\infty_T L^2_x}  \; .
  \end{equation}
 Finally, with \eqref{bilinearA}  and \eqref{zeA} in hand, Proposition \ref{proA1} follows from the classical linear estimates in Bourgain's spaces.
 \end{proof}
 Now the second step consists in proving the following estimate :
  \begin{proposition}\label{propo2A}
Let $ 0<\varepsilon<1 $,    $w \in C([0,T]; H^1(\R)) $  a  solution to \eqref{ww} with initial data $ \varphi $ and $ v\in Y^{1}_{\varepsilon,T}$. Then it holds
 \begin{equation}\label{est1propo2A}
 \|P_{\complement A_\varepsilon} w \|_{L^\infty_T L^2_x}^2 \lesssim \|P_{\complement A_\varepsilon} \varphi\|_{L^2}^2 +
 (\varepsilon^{1/2}+T^{1/4})
  \|w \|_{Y_{\varepsilon,T}^0}^2 \Bigl( \|v \|_{Y_{\varepsilon,T}^1} +  \|v \|_{Y_{\varepsilon,T}^1}^2\Bigr)
 \end{equation}
 where the implicit constant is independent of $ \varepsilon $.
 \end{proposition}
 \begin{proof}
 Applying the operator $ P_{\complement A_\varepsilon}  $ on \eqref{ww} 
 and taking the $ L^2_x $-scalar product  with  $  P_{\complement A_\varepsilon}  w $ we get
\begin{eqnarray*}
\frac{d}{dt} \|P_{\complement A_\varepsilon} w(t)\|_{L^2_x}^2  & = &  \int_{\R}  P_{\complement A_\varepsilon} \partial_x (v w)  P_{\complement A_\varepsilon} w \\
&= &   \int_{\R}  P_{\complement A_\varepsilon} \partial_x ( w P_{B_\varepsilon} v)  P_{\complement A_\varepsilon} w
+  \int_{\R}  P_{\complement A_\varepsilon} \partial_x ( w P_{\complement B_\varepsilon} v)  P_{\complement A_\varepsilon} w\\
& = & I_1+I_2 \;.
\end{eqnarray*}
Using the following commutator estimate (see for instance \cite{KT1}) 
$$
\| [P_{\complement A_\varepsilon} \partial_x, f]  g \|_{L^2_x} \lesssim \| \partial_x f\|_{L^\infty_x} \| g\|_{L^2_x} \; ,
$$
 and integrating by parts, we get 
\begin{eqnarray*}
I_1 &= &   \int_{\R}   P_{B_\varepsilon} v P_{\complement A_\varepsilon}  w_x P_{\complement A_\varepsilon} w
+    \int_{\R}  \Bigl( [P_{\complement A_\varepsilon}\partial_x,   P_{B_\varepsilon} v]    w \Bigr) P_{\complement A_\varepsilon} w\\
& \lesssim  & \|  \partial_x P_{B_\varepsilon} v\|_{L^\infty} \|w\|_{L^2_x}^2 \; .
\end{eqnarray*}
By the frequency projections, we easily control $ I_2 $ by 
\begin{eqnarray*}
I_2 & \lesssim &  \varepsilon^{-1/2}Ê \Bigl\|  P_{\complement A_\varepsilon} (
w P_{\complement B_\varepsilon} v\Bigr) \Bigr\|_{L^1_x}    \| P_{\complement A_\varepsilon}  w \|_{L^\infty_x} \\
& \lesssim & \| P_{\complement A_\varepsilon}  w\|_{L^\infty_x} \|w\|_{L^2_x}Ê  \| v\|_{H^1_x }Ê\; .
\end{eqnarray*}
Gathering the above estimates we infer that 
$$
\frac{d}{dt} \|P_{\complement A_\varepsilon} w(t)\|_{L^2_x}^2 \lesssim \Bigl((\|w(t)\|_{L^2_x}+\| P_{\complement A_\varepsilon}  w(t)\|_{L^\infty_x})(
\|  \partial_x P_{B_\varepsilon} v(t)\|_{L^\infty_x}+  \| v(t)\|_{H^1 })Ê\Bigr)\|w(t)\|_{L^2_x} \;  .
$$
On the other hand, applying Lemma \ref{carlos} on \eqref{ww} we get 
$$
\| P_{\complement A_\varepsilon} w\|_{L^{1}_T L^\infty_{x}} \lesssim (\varepsilon^{1/2}+T)
\|P_{\complement A_\varepsilon} \,w\|_{L^\infty_T\, L^2_{x}}
+ T \|v\|_{L^\infty_{T} H^1_{x}}   \|w\|_{L^\infty_T L^2_x} \quad .
$$
Therefore, integrating in time the next to the last inequality with \eqref{hypv2} in hand, leads to \eqref{est1propo2A}
 \end{proof}
 Combining Propositions \ref{proA1} and \ref{propo2A} we infer that 
 $$
\|w\|_{Y_{\varepsilon,T}^0}^2\le C\,  \| \varphi\|_{L^2}^2 + C\, (\sqrt{\varepsilon}+T^{\frac{1}{4}})\| w\|_{Y_{\varepsilon,T}^0}^2\| v\|_{Y_{\varepsilon,T}^1}
 \Bigl( 1+\| v\|_{Y_{\varepsilon,T}^1}^3 \Bigr)\; .
$$
which yieds the desired result according to \eqref{hypv1}
\bibliographystyle{amsplain}

\end{document}